\documentclass[preprint,3p]{amsart}
\usepackage{wrapfig}
\usepackage{amssymb}

\usepackage{float}

\usepackage{amsmath,amsfonts,amsthm,amstext}
\usepackage[latin1]{inputenc}
\usepackage{fancybox}
\usepackage{niceframe}
\usepackage{array}
\usepackage{newlfont}
\usepackage{verbatim}
\usepackage[dvips]{psfrag}
\usepackage{color}
\usepackage{float}
\usepackage[scriptsize]{caption}
\usepackage{indentfirst}
\usepackage[normalem]{ulem}
\usepackage{pst-text}
\usepackage{graphicx}
\usepackage{graphics}
\usepackage{overpic}

\graphicspath{{./Figures/}}
\usepackage{wrapfig}
\newcommand{\R}{\ensuremath{\mathbb{R}}}

\newcommand{\N}{\ensuremath{\mathbb{N}}}
\newcommand{\Q}{\ensuremath{\mathbb{Q}}}
\newcommand{\Z}{\ensuremath{\mathbb{Z}}}

\newcommand{\s}{\Sigma}

\newcommand{\la}{\lambda}

\newcommand{\e}{\varepsilon}

\newcommand{\V}{\mathcal{V}}

\newcommand{\Cr}{\mathcal{C}^{r}}
\newcommand{\Xr}{\chi^{r}}
\newcommand{\Or}{\Omega^{r}}
\newcommand{\rn}[1]{\mathbb{R}^{#1}}
\newcommand{\U}{\mathcal{U}}
\newcommand{\er}{\mathcal{O}}
\newcommand{\p}{\varphi}
\newcommand{\ag}{\alpha}
\newcommand{\bg}{\beta}
\newcommand{\cg}{\gamma}
\newcommand{\dg}{\delta}
\newcommand{\sgn}{\textrm{sgn}}
\newcommand{\fix}{\textrm{Fix}}
\newtheorem {theorem} {Theorem} 
\newtheorem {proposition} {Proposition}
\newtheorem {corollary} {Corollary}
\newtheorem {lemma} {Lemma}
\newtheorem {definition} {Definition}
\newtheorem {remark} {Remark}

\newtheorem*{thmA}{Theorem A}
\newtheorem*{thmB}{Theorem B}
\newtheorem*{thmC}{Theorem C}
\newtheorem*{thmD}{Theorem D}
\newtheorem*{thmE}{Corollary E}

\definecolor{verde}{rgb}{0.0,0.5,0.0}
\definecolor{azul}{rgb}{0,0,128}
\definecolor{roxo}{rgb}{0.44,0.16,0.39}
\definecolor{vinho}{rgb}{0.5,0.0,0.13}
\definecolor{lilas1}{rgb}{0.6,0.33,0.73}
\definecolor{rosa}{rgb}{0.84,0.04,0.33}
\definecolor{mostarda}{rgb}{0.91,0.41,0.17}
\definecolor{mostarda2}{rgb}{1.0,0.66,0.07}

\begin{document}

\title[Generic Singularities of 3D Piecewise Smooth Dynamical Systems]{Generic Singularities of 3D Piecewise Smooth Dynamical Systems}

%

\author{Ot\'avio M. L. Gomide}
\address[OMLG]{Department of Mathematics, Unicamp, IMECC\\ Campinas-SP, 13083-970, Brazil}
\email{otaviomleandro@hotmail.com}

\author{Marco A. Teixeira}
\address[MAT]{Department of Mathematics, Unicamp, IMECC\\ Campinas-SP, 13083-970, Brazil}
\email{teixeira@ime.unicamp.br}


\maketitle

\begin{abstract}
The aim of this paper is to provide  a discussion on current directions of research involving typical singularities of $3D$ nonsmooth vector fields. A brief survey of known results is presented.\newline\indent
The main purpose of this work is to describe  the dynamical features of a fold-fold singularity in its most basic form and to give a complete and detailed proof of its local structural stability (or instability). In addition, classes of all topological types of a fold-fold singularity are intrinsically characterized. Such proof essentially follows firstly from some lines laid out by Colombo, Garc\'ia, Jeffrey, Teixeira and others and secondly offers a rigorous mathematical treatment under clear and crisp assumptions and solid arguments.\newline\indent
One should to highlight that the geometric-topological methods employed lead us to the completely mathematical understanding of the dynamics around a T-singularity. This approach lends itself to applications in generic bifurcation theory.  It is worth to say that such subject is still poorly understood in higher dimension.
\end{abstract}

\section{Introduction}
\label{sec:1}

\indent Certain aspects of the theory of nonsmooth vector fields has been mainly motivated by the study of vector fields near the boundary of a manifold. Concerning this topic, many authors provided results and techniques which have been very useful in piecewise-smooth systems. It is worthwhile to cite in the 2-dimensional case works from Andronov et al, Peixoto, Teixeira (see \cite{ALGM, PP,T1})  and in higher dimensions the works from Sotomayor and Teixeira, Vishik and Percell (See \cite{ST,V, PE}). In particular,
in \cite{V} (1972), Vishik provided a classification of generic points lying in the boundary of a manifold, through techniques from Theory of Singularities.

Many papers have contributed to the analysis and generic classification of singularities of 2D Filippov systems (Kuznetsov et al, Guardia et al, Kozlova among others, see \cite{GTS,K,KU}). Specifically with respect to the fold-fold singularity we point Ekeland (See \cite{E}) and Teixeira (See \cite{T5}). Regarding the $n$-dimensional problem, we point out the work from Colombo and Jeffrey (see \cite{JC1}) which analyzes an $n$-dimensional family having a two-fold singularity, nevertheless the generic classification for $n>2$ is much more complicated and still poorly understood.

As far as we know, the first approach where a generic 3D fold-fold singularity was studied was offered by Teixeira in \cite{T4} (1981) where one finds a discussion on some features of the first return mapping that occurs around this singularity. Maybe due to this fact, the invisible fold-fold singularity is known as T-singularity.

In \cite{F} (1988), Filippov provided a mathematical formalization of the theory of nonsmooth vector fields. In the last chapter of \cite{F}, Filippov studied generic singularities in $3D$ nonsmooth systems, and a systematic mathematical analysis of the behavior around a fold-fold singularity was officially arisen. However, most of proofs were only roughly sketched and would require a better explanation and interpretation. In particular, the proofs of the results concerning the fold-fold singularity were obscure and unfinished. Many works appeared lately trying to explain it (See \cite{J1,J3,J2,Ponce,T2}).

In \cite{T2}, Teixeira established necessary conditions for the structural stability of the fold-fold singularity and he proved that it is not a generic property. Nevertheless, the case of the invisible fold-fold point having a hyperbolic first return map was not understood. He also provided results concerning asymptotic stability.

In \cite{J1,J3,J2}, Jeffrey et al also studied the problem of the classification of the structural stability around a fold-fold singularity. More specifically, in \cite{J2}, the authors studied the behavior of a 2-parameter semi-linear model $Z_{\ag,\bg}$ having a T-singularity at $Z_{0,0}$. By studying the first return map explicitly, they have found countably many curves $\gamma_{k}$ in a region of the parameter space, where the topological type $\beta_{k}$ of a system in $\gamma_{k}$ satisfies $\beta_k \neq \beta_l$ provided $k \neq l$. Moreover, they predict the existence of classes of structural stability between the curves $\gamma_{k}$ in this region. Guided by these results, we show that in the considered region of the parameter space, a general Filippov system $Z$ having a T-singularity at $p$ always has a first return map with complex eigenvalues. It brings several consequences to the behavior of $Z$ around $p$, in particular, it produces a foliation of this region in the parameter space depending on the argument of the eigenvalues of $Z$ such that, two systems in different leaves are not topologically equivalent near the T-singularity, which means that there is no class of stability in this region of parameters. It provides a negative answer to the questions raised in \cite{J2} concerning the validity of the results for general Filippov systems around a T-singularity.

A 3D-fold-fold singularity is an intriguing phenomenon that has no counterparts in smooth systems, and the complete characterization of the local structural stability of a 3D-nonsmooth system around an elliptic fold-fold singularity has been an open problem over the last 30 years. In this work, we believe that all mathematical existing gaps were filled up and the precise statement of results and proofs were well established.

It is worth to mention that the methods and techniques used in this paper provide a solution from a geometric-topological point of view. In addition, we present a generic and qualitative  characterization of a fold-fold singularity, in order to clarify any fact concerning the generality of the results.

\section{Setting the Problem} \label{prel}
In what follows we summarize a rough overall description of the basic concepts and results in order to set the problem.

\subsection{Filippov Systems}

For simplicity, let $M$ be a connected bounded region of $\rn{3}$ and let $f:M\rightarrow \rn{}$ be a smooth function having $0$ as a regular value, therefore $\Sigma=f^{-1}(0)$ is a compact embedded codimension one submanifold of $M$ which splits it in the sets $M^{\pm}=\{p\in M; \pm f(p)>0\}$.

Denote the set of germs of vector fields of class $\Cr$ at $\s$ by $\Xr$. Endow $\Xr$ with the $\Cr$ topology and consider $\Or=\Xr\times \Xr$ with the product topology.

If $Z=(X,Y)\in\Or$ then, a \textbf{nonsmooth vector field} is defined in some neighborhood $V$ of $\s$ in $M$ as follows:

\begin{equation}
Z(p)=F(p)+\sgn(f(p)) G(p),
\end{equation}
where $F(p)=\frac{X(p)+Y(p)}{2}$ and $G(p)=\frac{X(p)-Y(p)}{2}$.
\begin{definition}
	The \textbf{Lie derivative} of $f$ in the direction of the vector field $X\in \Xr$ at $p\in\Sigma$ is defined by $Xf(p)=X(p)\cdot \nabla f(p)$. The \textbf{tangency set} of $X$ with $\s$ is given by $S_{X}=\{p\in\Sigma;\ Xf(p)=0\}$.
\end{definition}

If $X_{1},\cdots, X_{n}\in \Xr$, the higher order Lie derivatives are defined as:
$$X_{n}\cdots X_{1}f(p)=X_{n}(X_{n-1}\cdots X_{1}f)(p),$$
i.e. $X_{n}\cdots X_{1}f(p)$ is the Lie derivative of the smooth function $X_{n-1}\cdots X_{1}f$ in the direction of the vector field $X_{n}$ at $p$. In particular, $X^{n}f(p)$ denotes the Lie derivative $X_{n}\cdots X_{1}f(p)$, where $X_{i}=X$, for $i=1,\cdots,n$.

If $Z=(X,Y)\in\Or$, then the switching manifold $\Sigma$ generically splits into three distinct open regions:
\begin{itemize}
	\item Crossing Region: $\Sigma^{c}=\{p\in \Sigma;\ Xf(p)Yf(p)>0\};$
	\item Stable Sliding Region: $\Sigma^{ss}=\{p\in \Sigma;\ Xf(p)<0,\ Yf(p)>0\};$
	\item Unstable Sliding Region: $\Sigma^{us}=\{p\in \Sigma;\ Xf(p)>0,\ Yf(p)<0\}.$			
\end{itemize}	

Consider the \textbf{sliding region} of $Z$ as $\s^{s}=\s^{ss}\cup \s^{us}$.

\begin{figure}[H]
	\centering
	\bigskip
	\begin{overpic}[width=10cm]{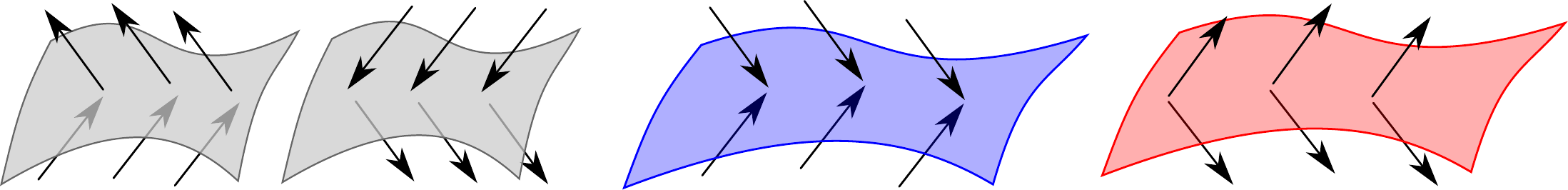}
		\put(16,-2){{\footnotesize $(a)$}}		
		\put(52,-2){{\footnotesize $(b)$}}		
		\put(83,-2){{\footnotesize $(c)$}}	
		\put(-3,5){{\footnotesize $\s$}}	
		\put(5,13){{\footnotesize $X$}}	
		\put(4,-1){{\footnotesize $Y$}}							
		
	\end{overpic}
	\bigskip
	\caption{Regions in $\s$: $\s^{c}$ in $(a)$, $\s^{ss}$ in $(b)$ and $\s^{us}$ in $(c)$.   }	
\end{figure}
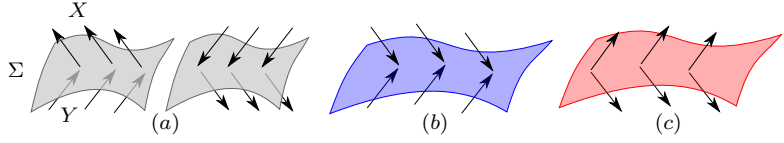 	

The tangency set of $Z$ will be referred as $S_{Z}=S_{X}\cup S_{Y}$. Notice that $\s$ is the disjoint union $\s^{c}\cup \s^{ss}\cup \s^{us}\cup S_{Z}$.

The concept of solution of $Z$ follows Filippov's convention. More details can be found in \cite{F,GTS,T8}.

We highlight that the local solution of $Z=(X,Y)\in \Or$ at a point $p\in \Sigma^{s}$ is given by the \textbf{sliding vector field}:
\begin{equation}
F_{Z}(p)=\frac{1}{Yf(p)-Xf(p)}\left(Yf(p)X(p)-Xf(p)Y(p)\right).
\end{equation}

\begin{remark}
	Notice that $F_{Z}$ is a vector field tangent to $\s^{s}$. The singularities of $F_{Z}$ in $\s^{s}$ are called \textbf{pseudo-equilibria} of $Z$.
\end{remark}

\begin{definition}\label{NSVF}
	If $p\in\s^{s}$, the \textbf{normalized sliding vector field} is defined by:
	\begin{equation}
	F_{Z}^{N}(p)=Y f(p)X (p)-X f(p)Y(p).
	\end{equation}	
\end{definition}

\begin{remark}
	If $R$ is a connected component of $\s^{ss}$, then $F_{Z}^{N}$ is a re-parameterization of $F_{Z}$ in $R$, then they have exactly the same phase portrait. If $R$ is a connected component of $\s^{us}$, then $F_{Z}^{N}$ is a (negative) re-parameterization of $F_{Z}$ in $R$, then they have the same phase portrait, but the orbits are oriented in opposite direction.
\end{remark}

If $Z=(X,Y)\in\Or$, consider all the integral curves of $X$ in $M^{+}$, all the integral curves of $Y$ in $M^{-}$ and the integral curves of $F_{Z}$ in $\s^{s}$. In this work, any oriented piecewise-smooth curve passing through $q$ is considered as a solution of $Z$ through $q$. 

\subsection{$\s$-Equivalence}

An orbital equivalence relation is defined in $\Or(M)$ as follows:

\begin{definition}\label{equivalencedef}
	Let $Z_{0},Z\in\Or$ be two germs of nonsmooth vector fields. We say that $Z_{0}$ is \textbf{topologically equivalent} to $Z$ at $p$ if there exist neighborhoods $U$ and $V$ of $p$ in $M$ and an order-preserving homeomorphism $h:U\rightarrow V$ such that it carries orbits of $Z_{0}$ onto orbits of $Z$, and it preserves $\s$, i.e. $h(\s\cap U)=\s\cap V$.
\end{definition}

The concept of local structural stability at a point $p\in\s$ is defined in the natural way.

\begin{definition}\label{sigmalss}
	$Z_{0}\in\Or$ is said to be \textbf{$\s$-locally structurally stable} if $Z_{0}$ is locally structurally stable at $p$, for each $p\in\s$.
	
	Denote the space of germs of nonsmooth vector fields $Z\in\Or$ which are $\s$-locally structurally stable by $\s_{0}$.
\end{definition}

\subsection{Reversible mappings}
The concepts in this section will be used in the sequel.

\begin{definition}
	A germ of \textbf{involution} at $0$ is a $\Cr$ germ of diffeomorphism $\varphi: \rn{2}\rightarrow \rn{2}$ such that $\varphi(0)=0$, $\varphi^{2}(x,y)=(x,y)$ and $\det[\varphi'(0,0)]=-1$.
\end{definition}

The set of germs of involutions at $0$ is denoted by $I^{r}$ and it is endowed with the $\Cr$ topology. Consider $W^{r}=I^{r}\times I^{r}$ endowed with the product topology.

\begin{definition}
	Let $\p=(\p_{0},\p_{1}),\ \psi=(\psi_{0},\psi_{1})\in W^{r}$ be two pairs of involutions at $0$. Then $\p$ and $\psi$ are said to be \textbf{topologically equivalent} at $0$ if there exists a germ of homeomorphism $h:(\rn{2},0)\rightarrow (\rn{2},0)$ which satisfies $h\p_{0}=\psi_{0}h$ and $h\p_{1}=\psi_{1}h$, simultaneously.
\end{definition}

The local structural stability of a pair of involutions in $W^{r}$ is defined in the natural way. The proof of the next theorem can be found in \cite{T4} such as more details about involutions.

\begin{theorem}\label{ssinv}
	A pair of involutions $(\p,\psi)$ is locally and simultaneous structurally stable at $0$ if and only if $0$ is a hyperbolic fixed point of the composition $\p\circ\psi$. Moreover, the structural stability in the space of pairs of involutions is not a generic property.
\end{theorem}

\section{Statement of the main results}

Define the following subsets of $\Or$:

\begin{itemize}
	\item $\s(G)$: $Z\in\Or$ such that each point $p\in\s$ is either a tangential singularity or a regular-regular point.
	
	\item $\s(R)$: $Z\in\Or$ such that for each regular-regular point $p\in\s$ of $Z$ we have either $p\in\s^{c}$ or $p\in\s^{s}$ and, in the second case, p is either a regular point or a hyperbolic singularity of $F_{Z}$;
	
	\item $\s(H)$: $Z\in\Or$ such that for each visible fold-fold point $p\in\s$, the normalized sliding vector field $F_{Z}^{N}$ has no center manifold in $\s^{s}$.
	
	\item $\s(P)$: $Z\in\Or$ such that for each invisible-visible point $p\in\s$, the normalized sliding vector field $F_{Z}^{N}$ is either transient in $\s^{s}$ or it has a hyperbolic singularity at $p$. Moreover, if $\phi_{X}$ is the involution associated to $Z$ then it satisfies:
	
	\begin{enumerate}
		\item $\phi_{X}(S_{Y})\pitchfork S_{Y}$ at $p$;
		
		\item $F_{Z}^{N}$ and $\phi_{X}^{*}F_{Z}^{N}$ are transversal at each point of $\s^{ss}\cap\phi_{X}(\s^{us})$;
		
		\item $\phi_{X}(S_{Y})\pitchfork F_{Z}^{N}$ in a neighborhood of $p$.
	\end{enumerate} 
	
	\item $\s(E)$:  $Z\in\Or$ such that for each T-singularity $p\in\s$, the first return map $\phi_{Z}$ associated to $Z$ has a fixed point at $p$ of type saddle with both local invariant manifolds $W^{u,s}_{loc}$ contained in $\s^{c}$.
\end{itemize}

\begin{remark}
	If $Z$ has a visible-invisible fold-fold singularity at $p$, then the roles of $X$ and $Y$ in the condition $\s(P)$ are interchanged.
\end{remark}


The main result of this work is the following theorem.

\begin{thmA}\label{thmA}
	$Z\in\Or$ is locally structurally stable at a T-singularity $p$ if and only if it satisfies condition $\s(E)$ at $p$.
\end{thmA}

The following theorem is proved in \cite{F,J1} and a detailed proof clarifying some obscure points is exhibited. 
\begin{thmB}\label{thmB}
	\begin{enumerate}
		\item[$i)$] $Z\in\Or$ is locally structurally stable at a hyperbolic fold-fold singularity $p$ if and only if it satisfies condition $\s(H)$ at $p$.
		\item[$ii)$] $Z\in\Or$ is locally structurally stable at a parabolic fold-fold singularity $p$ if and only if it satisfies condition $\s(P)$ at $p$.
	\end{enumerate}
\end{thmB}

\begin{thmC}\label{thmC}
	$\s_{0}=\s(G)\cap \s(R)\cap \s(H)\cap \s(P)\cap\s(E)$. 
\end{thmC}

\begin{thmD}\label{thmD}
	$\s_{0}$ is not residual in $\Or$.
\end{thmD}

As a corollary of the characterization Theorem \ref{thmC}, we obtain:

\begin{thmE}\label{thmE}
	\begin{enumerate}
		\item[$i)$] $\s_{0}$ is an open dense set in $\s(E)$. Moreover, $\s(E)$ is maximal with respect to this property.
		\item[$ii)$] If $Z\notin\s(E)$ then $Z$ has $\infty$-moduli of stability.
	\end{enumerate}	
\end{thmE}

In addition, if $Z$ has a T-singularity at $p$ and $\phi_{Z}$ has complex eigenvalues, then a neighborhood $\V$ of $Z$ in $\Or$ is foliated by codimension one submanifolds of $\Or$ corresponding to the value of the argument of the eigenvalues of the first return map. Moreover, the topological type along the corresponding leaf is locally constant.  

We conclude that the local behavior around a T-singularity implies in the non-genericity of $\s_{0}$ in $\Or$.

\section{Generic Singularities}
In this section, we provide a classification of the generic points of $\s$.

\begin{definition}
	Let $Z=(X,Y)\in\Or$, a point $p\in \s$ is said to be a \textbf{tangential singularity} of $Z$ if $Xf(p)Yf(p)=0$ and $X(p), Y(p)\neq 0$.
\end{definition}

\begin{definition}
	Let $Z=(X,Y)\in\Or$, a point $p\in \s$ is said to be a \textbf{$\s$-singularity} of $Z$ if $p$ is either a tangential singularity or a pseudo-equilibrium of $F_{Z}$. Otherwise, it is said to be a \textbf{regular-regular} point of $Z$
\end{definition}

\begin{definition}
	Let $Z=(X,Y)\in\Or$. A tangential singularity $p\in \s$ is said to be \textbf{elementary} if it satisfies one of the following conditions:
	\begin{enumerate}
		\item[(FR) -] $Xf(p)= 0$, $X^{2}f(p)\neq 0$ and $Yf(p)\neq 0$ (resp. $Xf(p)\neq 0$, $Yf(p)=0$ and $Y^{2}f(p)\neq 0$). In this case, $p$ is said to be a \textbf{fold-regular} (resp. regular-fold) point of $\s$.	
		\item[(CR) -] $Xf(p)= 0$, $X^{2}f(p)=0$, $X^{3}f(p)\neq 0$ and $Yf(p)\neq 0$ (resp. $Xf(p)\neq 0$, $Yf(p)=0$, $Y^{2}f(p)= 0$ and $Y^{3}f(p)\neq 0$), and $\{df(p),dXf(p),dX^{2}f(p)\}$ (resp. $\{df(p),$ $dYf(p),dY^{2}f(p)\}$) is a linearly independent set. In this case, $p$ is said to be a \textbf{cusp-regular} (resp. regular-cusp) point of $\s$.	
		\item[(FF) -] If $Xf(p)= 0$, $X^{2}f(p)\neq0, Yf(p)= 0$, $Y^{2}f(p)\neq 0$ and $S_{X}\pitchfork S_{Y}$ at $p$. In this case, $p$ is said to be a \textbf{fold-fold} point of $\s$. 			
	\end{enumerate}
\end{definition}

\begin{definition}
	Define $\Xi_{0}$ as the set of all germs of nonsmooth vector fields $Z\in\Or$ such that, for each $p\in\s$, either $p$ is a regular-regular point of $Z$ or $p$ is an elementary tangential singularity.
\end{definition}

From \cite{V}, we derive the following result:
\begin{proposition}\label{vishik}
	$\Xi_{0}$ is an open dense set of $\Or$.
\end{proposition}


In order to classify $\s_{0}$, we assume, without loss of generality, that $p$ is either a regular-regular point or an elementary tangential singularity.

The next step is devoted to characterize the locally structurally stable systems at generic singularities.

\begin{lemma} \label{sliding}
	Let $Z=(X,Y) \in \Omega^{r}$ and assume that $R$ is a connected component of $\Sigma^{s}$. Then:
	\begin{enumerate}
		\item The sliding vector field $F_{Z}$ is of class $\mathcal{C}^{r}$ and it can be smoothly extended beyond the boundary of $R$.
		\item If $p\in \partial R$ is a fold-regular point of $Z$, then $F_{Z}$ is transverse to $\partial R$ at $p$.
		\item If $p\in \partial R$ is a cusp-regular point of $Y$, then $F_{Z}$ has a quadratic contact with $\partial R$ at $p$.	
	\end{enumerate} 
\end{lemma}

This result is proved in \cite{T2}. It is a very useful tool to construct topological equivalences.

\begin{theorem}\label{generic}
	Let $Z=(X,Y)\in\Or$, then:\begin{enumerate}
		\item $Z$ is locally structurally stable at a regular-regular point $p\in\s$ if and only if $Z$ satisfies $\s(R)$ at $p$.
		\item $Z$ is locally structurally stable at any fold-regular singularity $p\in\s$.
		\item $Z$ is locally structurally stable at any cusp-regular singularity $p\in\s$.		
	\end{enumerate}
\end{theorem}

The proof of this result can be found in \cite{F, GTS}.

\section{Fold-Fold Singularity}
\subsection{A Normal Form}\label{NF_sec}

In this section we derive a normal form to study the fold-fold singularity and we present some consequences. This section is mainly motivated by the normal form of a fold point obtained by S. M. Vishik in \cite{V} and some variants such as \cite{J1,F,J2}.

\begin{proposition}\label{FFNF_prop}
	
	If $Z=(X,Y)\in\Or$ is a nonsmooth vector field having a fold-fold point at $p$ such that $S_{X}\pitchfork S_{Y}$ at $p$, then there exists coordinates $(x,y,z)$ around $p$ such that $f(x,y,z)=z$ and $Z$ is given by:
	
	\begin{equation}\label{FFNF}
	X(x,y,z)=\left(\begin{array}{c}
	\alpha\\
	1\\
	\dg y
	\end{array}\right)\textrm{ and }Y(x,y,z)=\left(\begin{array}{c}
	\gamma +\er(|(x,y,z)|)\\
	\beta +\er(|(x,y,z)|)\\
	x +\er(|(x,y,z)|^{2})
	\end{array}\right),
	\end{equation}
	where $\dg=\sgn(X^{2}f(p))$, $\sgn(\gamma)=\sgn(Y^{2}f(p))$, $\ag,\bg,\cg\in\rn{}$. 	
\end{proposition}
\begin{proof}[Outline]
	Use the coordinates $(x,y,z)$ of Theorem $2$ from \cite{V} to put $X$ in the form $X(x,y,z)=(0,1,\dg y)$ and $f(x,y,z)=z$. Now, consider the Taylor expansion of $Y$ in this coordinate system and perform changes to put $Yf(x,y,z)=x +\er(|(x,y,z)|^{2})$.
\end{proof}

\begin{definition}
	If $Z\in\Or$ has a fold-fold singularity at $p$, then the coordinate system of Proposition \ref{FFNF_prop} will be called \textbf{normal coordinates} of $Z$ at $p$ and the parameters of $Z$ in the normal coordinates will be referred as \textbf{normal parameters} of $Z$ at $p$. Denote $Z=Z(\ag,\bg,\cg)$.
\end{definition}

\begin{remark}
	If $\gamma=\pm1$, $\alpha=V^{+}$ and $\beta=V^{-}$, then this normal form and the model used in \cite{J1,J3,J2}, have the same semi-linear part. Geometrically, $V^{+}$ ($V^{-}$) measures the cotangent of the angle $\theta^{+}$ ($\theta^{-}$) between $X(0)$ ($Y(0)$) and the fold line $S_{X}$ ($S_{Y}$). See \cite{J3} for more details.
\end{remark}

\begin{corollary}\label{formaS}
	If $Z=(X,Y)\in\Or$ is a nonsmooth vector field having a fold-fold point at $p$ such that $S_{X}\pitchfork S_{Y}$ at $p$, then there exist coordinates $(x,y,z)$ around $p$ defined in a neighborhood $U$ of $p$ in $M$, such that:
	\begin{enumerate}
		\item $f(x,y,z)=z$;
		\item $S_{X}\cap U=\{(x,0,0);\ x\in(-\e,\e)\}$, for $\e>0$ sufficiently small;
		\item $S_{Y}\cap U=\{(g(y),y,0);\ y\in(-\e,\e)\}$, for $\e>0$ sufficiently small, where $g$ is a $\Cr$ function such that $g(y)=\er(y^{2})$, i.e., $S_{Y_{0}}$ is locally a smooth curve tangent to the $y$-axis.
	\end{enumerate} 
\end{corollary}
\begin{proof}[Outline of the Proof]
	It follows directly from Proposition \ref{FFNF_prop} and the Implicit Function Theorem. 
\end{proof}

\begin{proposition}\label{SVF_FF}
	
	Let $Z=(X,Y)\in\Or$ be a nonsmooth vector field having a fold-fold point at $p$ such that $S_{X}\pitchfork S_{Y}$ at $p$. Then, the normalized sliding vector field of $Z$ has a singularity at $p$ and it is given by
	
	\begin{equation}\label{foldNF2}
	F_{Z}^{N}(x,y)=\left(\begin{array}{cc}
	\alpha & -\delta \gamma \\
	1 & -\delta \beta
	\end{array}\right)\cdot \left(\begin{array}{c}
	x\\
	y
	\end{array}\right)  +\er(|(x,y)|^{2}),
	\end{equation}
	in the normal coordinates of $Z$ at $p$, where $\dg=\sgn(X^{2}f(p))$, $\sgn(\gamma)=\sgn(Y^{2}f(p))$, $\ag,\bg,\cg\in\rn{}$. 	
\end{proposition}
\begin{proof}[Outline of the Proof]
	It follows directly from the expression of $Z$ in this coordinate system.
\end{proof}

Finally, we can classify a fold-fold singularity in four topologically distinct classes:

\begin{definition}
	A fold-fold point $p$ of $Z=(X,Y)\in\Or$ is said to be:
	\begin{itemize}
		\item a \textbf{visible fold-fold} if $X^{2}f(p)>0$ and $Y^{2}f(p)<0$;
		\item an \textbf{invisible-visible fold-fold} if $X^{2}f(p)<0$ and $Y^{2}f(p)<0$;	
		\item a \textbf{visible-invisible fold-fold} if $X^{2}f(p)>0$ and $Y^{2}f(p)>0$;			
		\item an \textbf{invisible fold-fold} if $X^{2}f(p)<0$ and $Y^{2}f(p)>0$, in this case, $p$ is also called a \textbf{T-singularity}.				
	\end{itemize}
\end{definition}

\begin{remark}
	Notice that the visible-invisible case can be obtained from the invisible-visible one by performing an orientation reversing change of coordinates. Also, we refer a visible, invisible-visible/visible-invisible, invisible as a hyperbolic, parabolic, elliptic fold-fold, respectively.
\end{remark}

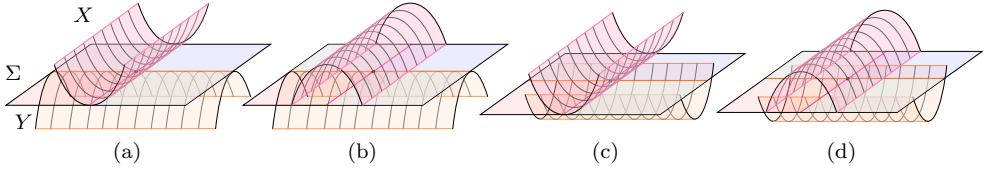
\begin{figure}[H]
	\centering
	\bigskip
	\begin{overpic}[width=13cm]{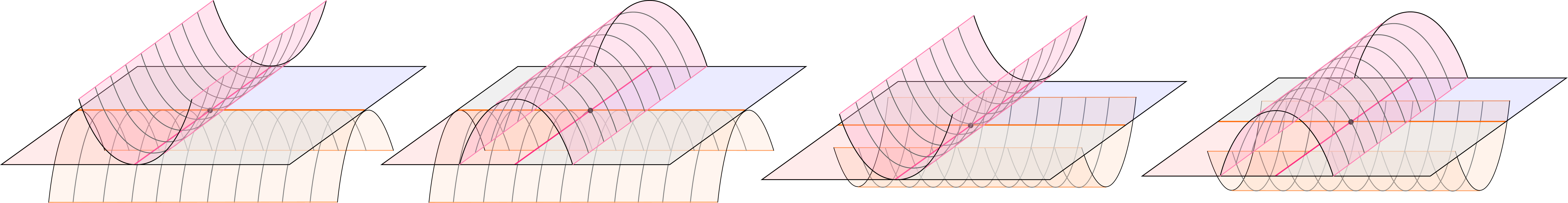}
		\put(11,-3){{\footnotesize (a)}}						
		\put(35,-3){{\footnotesize (b)}}
		\put(60,-3){{\footnotesize (c)}}	
		\put(84,-3){{\footnotesize (d)}}	
		\put(0,5){{\footnotesize $\s$}}		
		\put(7,11){{\footnotesize $X$}}	
		\put(1,0){{\footnotesize $Y$}}																					
	\end{overpic}
	\bigskip
	\caption{Fold-Fold Singularity: (a) Hyperbolic, (b,c) Parabolic and (d) Elliptic.}	
\end{figure}

\subsection{Sliding Dynamics}\label{slidingsection}

In this subsection we discuss the sliding dynamics around a fold-fold singularity. This is a matured topic which has been well developed in \cite{J3,F}. 

From Proposition \ref{FFNF_prop} and Lemma \ref{sliding}, we already know the behavior of the sliding vector field near a fold-fold singularity in a generic scenario (not only for the truncated system).

Let $Z=Z(\ag,\bg,\cg)\in\Or$ having a fold-fold singularity at $p$, and consider its normalized sliding vector field $F_{Z}^{N}$ in normal coordinates.

Consider:
$$
\begin{array}{l}
R_{E}^{1}= \{(\ag,\bg,\cg)\in \rn{2}\times \R^{+};\ \ag\bg>\cg \textrm{ and }\ag<0,\ \bg<0\} \\

R_{E}^{2}= \rn{2}\times \R^{+}\setminus \overline{R_{I}^{1}}\\

R_{H}^{1}= \{(\ag,\bg,\cg)\in \rn{2}\times \R^{-};\ \ag\bg<\cg \textrm{ and }\ag>0,\ \bg<0\} \\

R_{H}^{2}= \rn{2}\times \R^{-}\setminus \overline{R_{V}^{1}}\\

R_{P}^{1}= \{(\ag,\bg,\cg)\in \rn{2}\times \R^{-};\ \ag\bg<\cg \textrm{ and }\bg-\ag>-2\sqrt{-\cg}\}\\

R_{P}^{2}= \{(\ag,\bg,\cg)\in \rn{2}\times \R^{-};\ \ag\bg<\cg \textrm{ and } \ag>0\}\\

R_{P}^{3}= \{(\ag,\bg,\cg)\in \rn{2}\times \R^{-};\ \ag\bg>\cg,\ \bg+\ag>0 \textrm{ and } \bg-\ag<-2\sqrt{\cg}\}\\

R_{P}^{4}= \{(\ag,\bg,\cg)\in \rn{2}\times \R^{-};\ \ag\bg>\cg,\ \bg+\ag<0 \textrm{ and } \bg-\ag<-2\sqrt{-\cg}\}
\end{array}
$$

We claim that:

\textbf{Claim 1:} If $p$ is an elliptic fold-fold singularity and $(\ag,\bg,\cg)\in R^{1}_{E}$ then $F_{Z}$ has an invariant manifold $W$ in $\s^{s}$ passing through $p$ and each orbit of $F_{Z}$ is transverse to $S_{Z}$ and reaches $p$ asymptotically to $W$ (for a finite positive time in $\s^{ss}$ and negative time in $\s^{us}$).

\textbf{Claim 2:} If $p$ is an elliptic fold-fold singularity and $(\ag,\bg,\cg)\in R^{2}_{E}$ then $F_{Z}$ has an invariant manifold $W$ in $\s^{s}$ passing through $p$ and each orbit is transverse to $S_{Z}$ and does not reach $p$, with exception of $W$.

\textbf{Claim 3:} If $p$ is a hyperbolic fold-fold singularity and $(\ag,\bg,\cg)\in R^{1}_{H}$ (resp. $(\ag,\bg,\cg)\in R^{2}_{H}$ ) then $F_{Z}$ is of the same type of claim $1$ (resp. claim $2$) for reverse time.

\textbf{Claim 4:} If $p$ is a parabolic fold-fold singularity and $(\ag,\bg,\cg)\in R^{1}_{P}$ then each orbit in $\s^{ss}$ (resp. $\s^{us}$) is transverse to $S_{X}$ (resp. $S_{Y}$) and reaches $S_{Y}$ (resp. $S_{X}$) transversally for a positive finite time. In this case we say that $F_{Z}$ has transient behavior in $\s^{s}$.

\textbf{Claim 5:} If $p$ is a parabolic fold-fold singularity and $(\ag,\bg,\cg)\in R^{2}_{P}$  then there exist two invariant manifolds $W_{1}$ and $W_{2}$ in $\s^{s}$ passing through $p$ which divides $\s^{ss}$ (and $\s^{us}$) in three sectors. The intermediate sector is of hyperbolic type and in the other sectors the orbits are transversal to $S_{Z}$ and goes away from $p$ (the orientation of the orbits is given in Figure \ref{fig:sliding}).


\textbf{Claim 6:} If $p$ is a parabolic fold-fold singularity and $(\ag,\bg,\cg)\in R^{3}_{P}$  then there exist two invariant manifolds $W_{1}$ and $W_{2}$ in $\s^{s}$ passing through $p$ which divides $\s^{ss}$ in three sectors. In the intermediate sector each orbit reaches $p$ for a finite positive time asymptotically to $W_{1}$. In the left one each orbit is transverse to $S_{Y}$ and reaches $p$ for a finite positive time asymptotically to $W_{1}$. In the right one, each orbit is transverse to $S_{X}$ and goes away from $p$. The behavior in $\s^{us}$ is similar and can be seen in Figure \ref{fig:sliding}.

\textbf{Claim 7:} If $p$ is a parabolic fold-fold singularity and $(\ag,\bg,\cg)\in R^{4}_{P}$  then $F_{Z}$ has the same behavior as in claim $6$ for reverse time and changing the role of $W_{1}$ and $W_{2}$, $S_{X}$ and $S_{Y}$, right and left.

\textbf{Claim 8:} If $(\ag,\bg,\cg)$ is not in any of these regions then $F_{Z}$ presents bifurcations in $\s^{s}$.

All these claims can be straightforward verified by analyzing the linear part of the normalized sliding vector field $F_{Z}^{N}$. We omitted the proofs due to the limitation of space.

\begin{figure}[H] 
	\centering
	\bigskip
	\begin{overpic}[width=15cm]{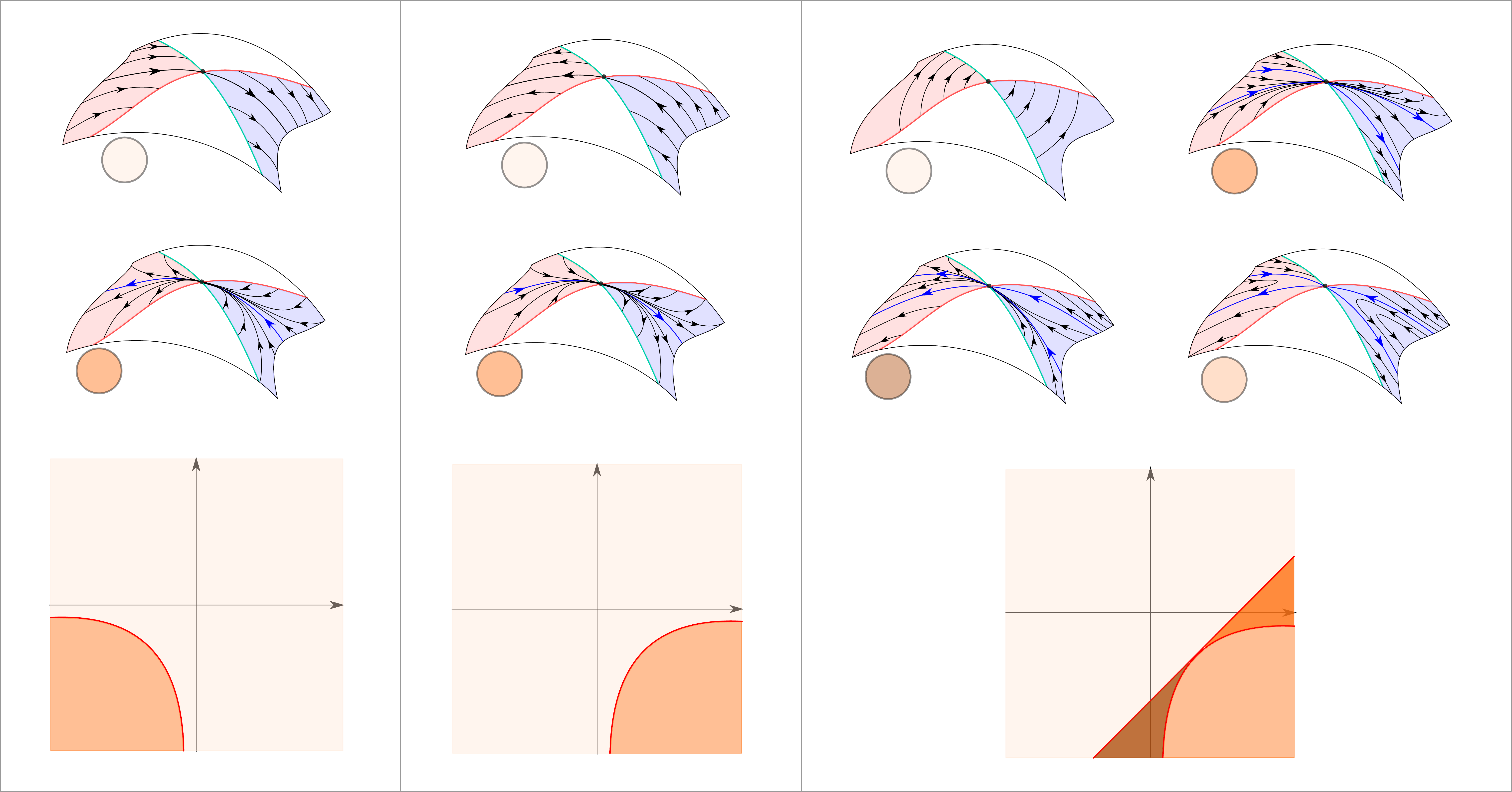}
		\put(12,-3){{\footnotesize (a)}}						
		\put(38,-3){{\footnotesize (b)}}
		\put(75,-3){{\footnotesize (c)}}
		\put(3,45){{\footnotesize $\s$}}
		\put(9,50.5){{\footnotesize $S_{X}$}}	
		\put(21,47){{\footnotesize $S_{Y}$}}
		\put(19.4,42){{\scriptsize $W$}}
		\put(19.4,27.9){{\scriptsize $W$}}	
		\put(45.7,27.9){{\scriptsize $W$}}
		\put(45.7,41.5){{\scriptsize $W$}}	
		\put(95.5,42.5){{\scriptsize $W_{1}$}}
		\put(93,40){{\scriptsize $W_{2}$}}
		\put(95.5,29){{\scriptsize $W_{1}$}}
		\put(92.8,26.5){{\scriptsize $W_{2}$}}	
		\put(73,29){{\scriptsize $W_{2}$}}
		\put(70.5,26.5){{\scriptsize $W_{1}$}}	
		\put(5.4,27.4){{\scriptsize $R_{E}^{1}$}}
		\put(6.8,41.3){{\scriptsize $R_{E}^{2}$}}	
		\put(31.7,27.2){{\scriptsize $R_{H}^{1}$}}
		\put(33.3,41.1){{\scriptsize $R_{H}^{2}$}}	
		\put(79.5,26.7){{\scriptsize $R_{P}^{2}$}}
		\put(80.3,40.5){{\scriptsize $R_{P}^{4}$}}			
		\put(57.5,27){{\scriptsize $R_{P}^{3}$}}			
		\put(59,40.5){{\scriptsize $R_{P}^{1}$}}	
		\put(23,12){{\scriptsize $\ag$}}	
		\put(49.5,11.5){{\scriptsize $\ag$}}	
		\put(86,11.5){{\scriptsize $\ag$}}										
		\put(12.5,22.5){{\scriptsize $\bg$}}	
		\put(39,22){{\scriptsize $\bg$}}	
		\put(76,22){{\scriptsize $\bg$}}																																																												
	\end{overpic}
	\bigskip
	\caption{Sliding dynamics near a fold-fold singularity of type elliptic (a), hyperbolic (b) and parabolic (c). In each case, the regions above are outlined in the $(\ag,\bg)$-parameter space for a fixed value of $\cg$.}\label{fig:sliding}	
\end{figure}
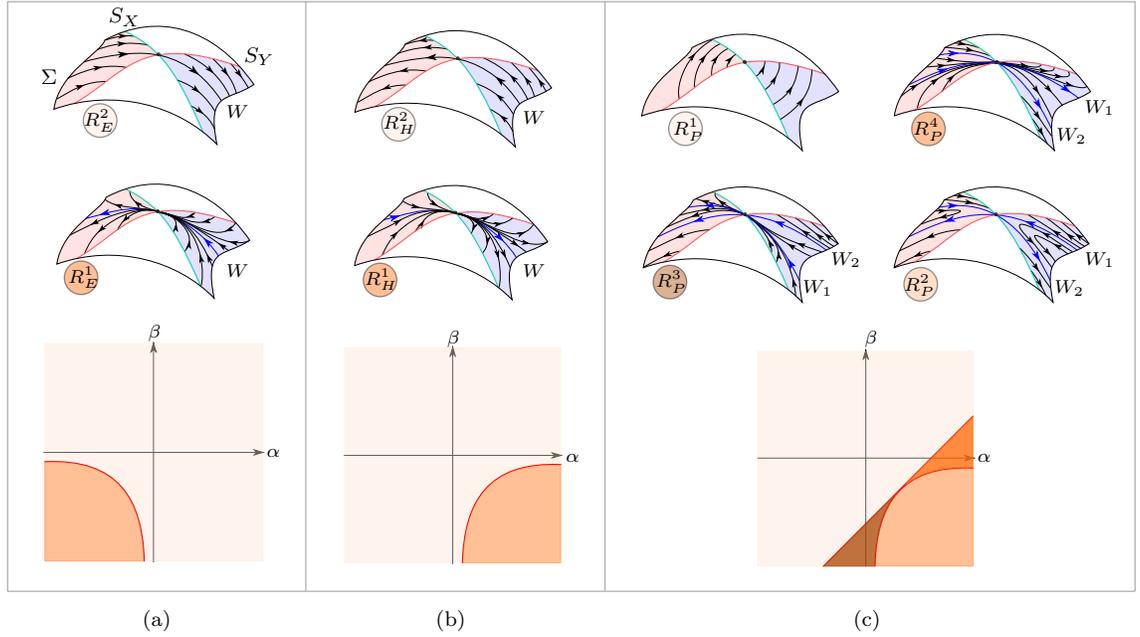 

\section{Proofs of Theorems A and D}

This section is devoted to prove Theorems A and D. In the sequel we develop some Lemmas and Propositions which will lead us to the proof of the Theorems.

Assume that $Z\in\Or$ has a T-singularity at $p$. Therefore, we have a first return map $\phi$ of $Z$ defined around $p$. In order to study the local structural stability of $Z$, it will be crucial to study the dynamics of $\phi$.
Now, we derive the existence and some properties of $\p$.

\begin{lemma}\label{inv}
	Let $Z=(X,Y)\in\Or$ be a nonsmooth vector field having a T-singularity at $p$ such that $S_{X}\pitchfork S_{Y}$ at $p$. There exist two involutions $\phi_{X}:(\s,p)\rightarrow(\s,p)$ and $\phi_{Y}:(\s,p)\rightarrow(\s,p)$ associated to the folds $X$ and $Y$ such that:
	\begin{itemize}
		\item $\fix(\phi_{X})=S_{X}$;
		\item $\fix(\phi_{Y})=S_{Y}$;
		\item $\phi=\phi_{X}\circ \phi_{Y}$ is a first return map of $Z$ such that $\phi(p)=p$.
	\end{itemize}
\end{lemma}

The proof of Lemma \ref{inv} can be found in \cite{BMT} (Lemma $1$). A straightforward verification shows the following results:

\begin{lemma} \label{comut}
	If $\phi=\p\circ\psi$, where $\p$ and $\psi$ are involutions of $\R^{2}$ at $0$, then $\phi^{n}\circ\p=\p\circ \phi^{-n}$ and $\psi\circ\phi^{n}=\phi^{-n}\circ\psi,$	for each $n\in\Z$.			
\end{lemma}

\begin{proposition}\label{reversibility}
	If $\phi=\p\circ\psi$, where $\p$ and $\psi$ are involutions of $\s$ at $p$, then the invariant manifolds $W^{s}$ and $W^{u}$ of $\phi$ at $p$ are interchanged by $\p$ and $\psi$ in the following way: $$\psi(W^{s})\subset W^{u} \textrm{ and }\p(W^{u})\subset W^{s}.$$
\end{proposition}

Now, using the normal coordinates of $Z=(X,Y)$ at an elliptic fold-fold singularity  we obtain the following expressions for the associated involutions. Notice that the involution $\phi_{X}$ is completely determined in these coordinates.

\begin{lemma}\label{forma-inv}
	Let $Z=(X,Y)\in\Or$ be a nonsmooth vector field having a T-singularity at $p$ such that $S_{X}\pitchfork S_{Y}$ at $p$. Consider the normal coordinates $(x,y,z)$ of $Z$ at $p$. Then:
	\begin{equation}
	\phi_{X}(x,y)=(x-2\ag y, -y) \textrm{ and } \phi_{Y}(x,y)=\left(
	-x,	-\frac{2\bg}{\cg}x+y\right)+h.o.t.,
	\end{equation}
	in these coordinates, where $\ag,\bg,\cg$ are the normal parameters of $Z$ at $p$.
\end{lemma}

Finally, we associate the local structural stability of $Z$ at an elliptic fold-fold singularity with the local structural stability of the pair of involutions associated to $Z$.

\begin{lemma}\label{rel}
	Let $Z_{0}=(X_{0},Y_{0})\in\Or$ such that $p$ is a T-singularity for $Z_{0}$. If $Z_{0}$ is locally structurally stable at $p$ in $\Or$ then the pair of involutions $(\phi_{X_{0}},\phi_{Y_{0}})$ associated to $Z_{0}$ is locally and simultaneous structurally stable at $0$ in $W^{r}$.
\end{lemma}	
\begin{proof}
	In fact, since $p$ is a T-singularity of $Z_{0}$, there exist neighborhoods $\V$ of $Z_{0}$ in $\Or$ and $V$ of $p$ in $M$ such that, each $Z\in\V$ has a unique Teixeira singularity at $q(Z)\in V\cap \s$.
	
	Consider the map $F:\V\rightarrow W^{r}$ given by:
	\begin{equation}\label{expF}
	F(X,Y)=(\phi_{X},\phi_{Y}),
	\end{equation}
	where $\phi_{X}$ and $\phi_{Y}$ are the involutions at $(0,0)$ of $\rn{2}$ associated to $X$ and $Y$, respectively.
	
	From the continuous dependence of solutions with respect to initial conditions and parameters, it follows that $F$ is a continuous map.
	
	Moreover, there exists a neighborhood $\U$ of $(\phi_{X_{0}},\phi_{Y_{0}})$ in $W^{r}$, such that, for each $(\tau,\psi)\in\U$, there exists a vector field $Z=(X,Y)\in\V$ such that $\tau=\phi_{X}$ and $\psi=\phi_{Y}$, and it can be done in a continuous fashion. 
	
	Then, reducing $\V$ if necessary, it follows that $F:\V\rightarrow W^{r}$ is an open continuous map.
	
	Since $Z_{0}$ is locally structurally stable at $p$ in $\Or$, $\V$ can be reduced such that every $Z\in\V$ is topologically equivalent to $Z_{0}$. 
	
	Now, if $Z\in \V$, there exists a topological equivalence $h:V_{1}\rightarrow V_{2}$ between $Z_{0}$ and $Z$, where $V_{1}$ and $V_{2}$ are neighborhoods of $p$ in $M$ and $q(Z)\in V_{2}$. 
	
	In particular, it induces a homeomorphism $h:\s\cap V_{1}\rightarrow \s\cap V_{2}$ such that $h(p)=q(Z)$. Using coordinates, $(x,y,z)$ around $p$ and $(u,v,w)$ around $q(Z)$ such that $f(x,y,z)=z$ and $f(u,v,w)=w$, the induced homeomorphism $h$ can be seen as $h:U_{1}\rightarrow U_{2}$, where $U_{1}$ and $U_{2}$ are neighborhoods of $(0,0)$ in $\rn{2}$ and $h(0,0)=(0,0)$.
	
	Let $(x,y)\in\s-S_{X_{0}}$, then $\phi_{X_{0}}(x,y)=\varphi_{X_{0}}(t(x,y),(x,y,0))$. Now, since $h$ is a topological equivalence, it follows that:
	$$h(\phi_{X_{0}}(x,y))=\varphi_{X}(t_{X}, h(x,y))).$$
	
	Now, since $(x,y)\in\s-S_{X_{0}}$, it follows that $h(x,y)\in\s-S_{X}$, which means that $\phi_{X}(h(x,y))=\varphi_{X}(t(h(x,y)),h(x,y))$, with $t(h(x,y))\neq 0$.
	
	Notice that $h(x,y)\neq h(\phi_{X_{0}}(x,y))$, since $h$ is a homeomorphism and $(x,y)\neq \phi_{X_{0}}(x,y)$, which means that $\varphi_{X}(t_{X},h(x,y))\in \s$ and $\varphi_{X}(t_{X},h(x,y))\neq h(x,y)$. 
	By uniqueness of $t(h(x,y))\neq 0$, it follows that $t(h(x,y))=t_{X}$.
	
	Hence,
	\begin{equation} \label{equiv}
	h(\phi_{X_{0}}(x,y))=\phi_{X}(h(x,y)).
	\end{equation}
	
	It is trivial to see that \ref{equiv} is also true when $(x,y)\in S_{X_{0}}$, by observing that $h(S_{X_{0}})=S_{X}$. Hence $h$ is an equivalence between the involutions $\phi_{X_{0}}$ and $\phi_{X}$.
	
	Analogously, by changing the roles of $X$ and $Y$, it can be shown that $h$ is also an equivalence between the involutions $\phi_{Y_{0}}$ and $\phi_{Y}$.
	
	We conclude that $h$ is a (simultaneous) topological equivalence between the pairs of involutions $(\phi_{X_{0}},\phi_{Y_{0}})$ and $(\phi_{X},\phi_{Y})$.
	
	Since $Z$ is arbitrary in $\V$, it follows that every pair of involutions in $\U$ is topologically equivalent to $(\phi_{X_{0}},\phi_{Y_{0}})$, and since $\U$ is open in $W^{r}$, it follows that $(\phi_{X_{0}},\phi_{Y_{0}})$ is local and simultaneous structurally stable in $W^{r}$.
\end{proof}

The following result is obtained by combining Theorem \ref{ssinv}  and Lemma \ref{rel}.

\begin{proposition}\label{rel2}
	Let $Z_{0}\in\Or$ having a T-singularity at $p$, and let $(\phi_{X_{0}},\phi_{Y_{0}})$ be the pair of involutions of $\rn{2}$ at $(0,0)$ associated to $Z_{0}$. If $0$ is not a hyperbolic fixed point of $\phi_{Y_{0}}\circ\phi_{X_{0}}$, then $Z_{0}$ is locally structurally unstable at $p$.
\end{proposition}	

A simple computation of eigenvalues and eigenvectors allows us to study the fixed point $p$ of the first return map $\phi$:

\begin{lemma}\label{posinv}
	Let $Z=(X,Y)\in\Or$ be a nonsmooth vector field having a T-singularity at $p$ such that $S_{X}\pitchfork S_{Y}$ at $p$. Let $(\ag,\bg,\cg)$ be the normal parameters of $Z$ at $p$.
	\begin{enumerate}
		\item  If $\ag\bg(\ag\bg-\cg)\leq 0$, then $0$ is not a hyperbolic fixed point of $\phi$. In addition, if $\ag\bg(\ag\bg-\cg)< 0$, then $\phi$ has complex eigenvalues.
		\item  If $\ag\bg(\ag\bg-\cg)> 0$, then $0$ is a saddle point of $\phi$. In addition, if $\la,\mu$ are the eigenvalues of $\phi$ such that $|\mu|<1<|\la|$, and $v_{\mu},v_{\la}$ are the correspondent eigenvectors, then:
		\begin{enumerate}
			\item  If $\ag>0$ and $\bg>0$, then $v_{\mu},v_{\la}\in \s^{s}$. 
			\item  If $\ag>0$ and $\bg<0$, then $v_{\mu}\in\s^{c}$ and $v_{\la}\in \s^{s}$.
			\item  If $\ag<0$ and $\bg>0$, then $v_{\mu}\in\s^{s},$ and $v_{\la}\in \s^{c}$.		
			\item  If $\ag<0$ and $\bg<0$ then $v_{\mu},v_{\la}\in \s^{c}$.							
		\end{enumerate}
	\end{enumerate}
\end{lemma}

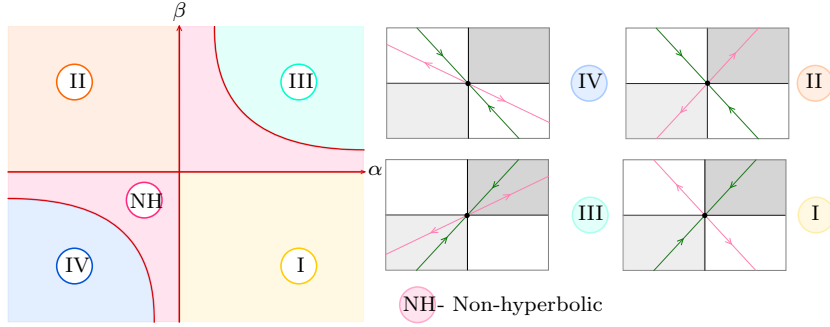
\begin{figure}[H]
	\centering
	\begin{overpic}[width=11cm]{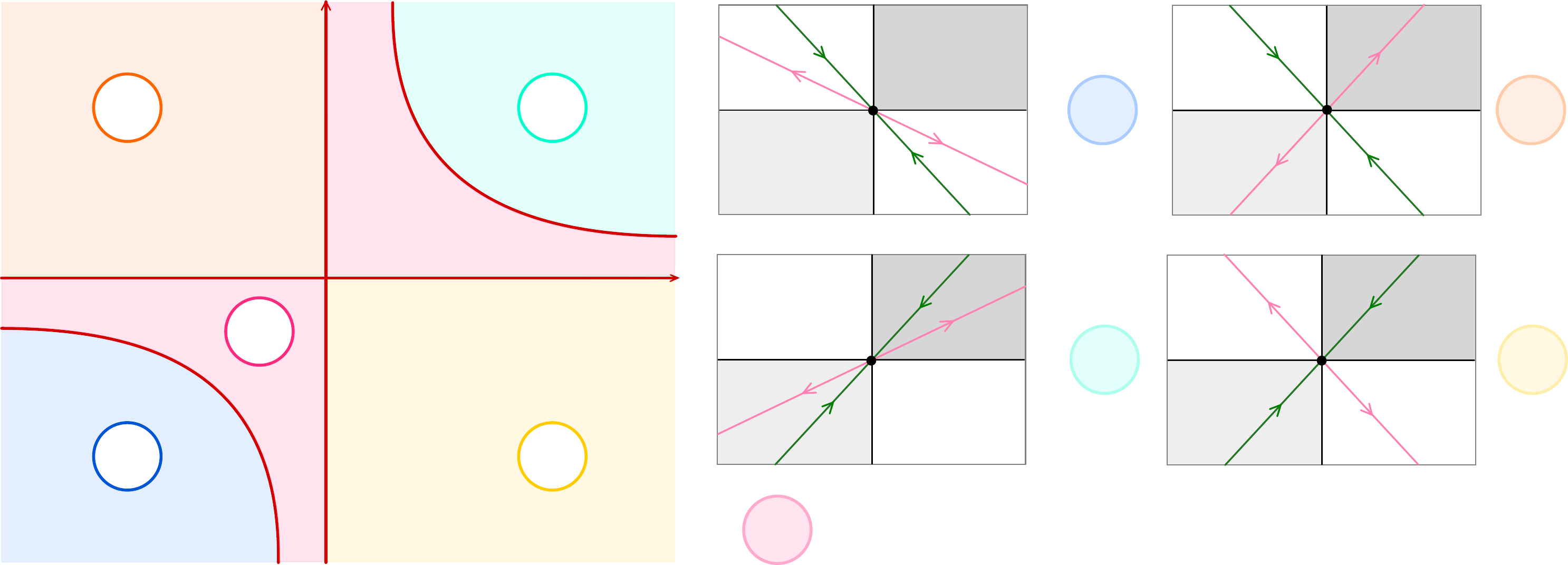}
		\put(35,6.2){{\footnotesize I}}	
		\put(7,6.2){{\footnotesize IV}}		
		\put(7.5,28.5){{\footnotesize II}}	
		\put(34,28.5){{\footnotesize III}}		
		\put(14.8,14){{\footnotesize NH}}		
		\put(43.6,18){{\footnotesize $\alpha$}}	
		\put(20,37){{\footnotesize $\beta$}}					
		\put(97.2,12.5){{\footnotesize I}}	
		\put(69,28,5){{\footnotesize IV}}		
		\put(96.5,28.5){{\footnotesize II}}	
		\put(69,12.5){{\footnotesize III}}		
		\put(47.8,1.5){{\footnotesize NH}}
		\put(52,1.5){{\footnotesize - Non-hyperbolic}}	
	\end{overpic}
	\caption{Regions of the $(\ag,\bg)$-parameter space with the corresponding behavior of the first return map $\p$, for a fixed value of $\cg>0$.}\label{TI-bifdiagram}
\end{figure}

\begin{proposition}\label{NH-unst}
	Let $Z_{0}=(X_{0},Y_{0})\in\Or$ be a germ of nonsmooth vector field having a T-singularity at $p$. Let $(\ag,\bg,\cg)$ be the normal parameters of $Z_{0}$ at $p$. If $\ag\bg(\ag\bg-\cg)\leq 0$, then $Z_{0}$ is locally structurally unstable at $p$.
\end{proposition}
\begin{proof}
	It follows directly from Proposition \ref{rel2} and the fact that $p$ is not a hyperbolic fixed point of the first return map $\phi_{0}=\phi_{X_{0}}\circ\phi_{Y_{0}}$ associated to $Z_{0}$.	In the sequel we present an explicitly argument for the local structural instability of $Z_{0}$. It is mainly based on \cite{BT} and the Blow-up procedure (see \cite{A}).
	
	Let $\phi_{0}:(\s,p)\rightarrow(\s,p)$ be the (germ of) first return map associated to $Z_{0}$ at $p$. From the conditions assumed in the Theorem, it follows that $\phi_{0}$ has eigenvalues $\la_{\pm}=a\pm i b$, where $a^{2}+b^{2}=1$. Using the normal form of $Z_{0}$ and basic linear algebra, it is easy to find coordinates $(x,y)$ of $\s$ at $p$, such that:
	$$\phi_{0}(x,y)=(ax-by,bx+ay)+\er(|(x,y)|^{2}).$$
	
	Consider the germs of functions $h_{1},h_{2}:(\R^{2},0)\rightarrow (\R^{2},0)$, given by:
	$$h_{1}(x,y)=(x,y) \textrm{ and } h_{2}(x,y)=\sqrt{x^{2}+y^{2}}(x,y).$$
	
	Notice that $h_{1},h_{2}$ are germs of homeomorphisms if we exclude the origin in their domains.
	
	If $(x,y)\neq (0,0)$, a straightforward computation shows that:
	
	$$\psi_{0}(x,y)= h_{2}^{-1}\circ \phi_{0}\circ h_{1}(x,y)= \frac{1}{\sqrt{x^{2}+y^{2}}}\phi_{0}(x,y).$$
	
	Therefore, $\phi_{0}$ and $\psi_{0}$ are topologically equivalent. Identifying $(x,y)=x+iy$ and writing $\psi_{0}$ in polar coordinates, we obtain:
	
	$$\psi_{0}(re^{i\theta})=e^{i(\theta +\tau)}+r\Gamma(r,\theta),$$
	where $a+ib=e^{i\tau}$ and $\Gamma$ is a bounded function.
	
	Hence, $\psi_{0}$ is the \textbf{blow-up} of $\phi_{0}$ at the origin. If $r\rightarrow 0$, $\psi_{0}$ induces a diffeomorphism $\overline{\psi_{0}}:S^{1}\rightarrow S^{1}$, given by:
	$$\overline{\psi_{0}}(\theta)=\theta +\tau,$$
	and the singularity of $\phi_{0}$ is brought to the circle $S^{1}$ with the dynamics induced by $\overline{\psi_{0}}$.
	
	\begin{figure}[H]
		\centering
		\bigskip
		\begin{overpic}[width=6cm]{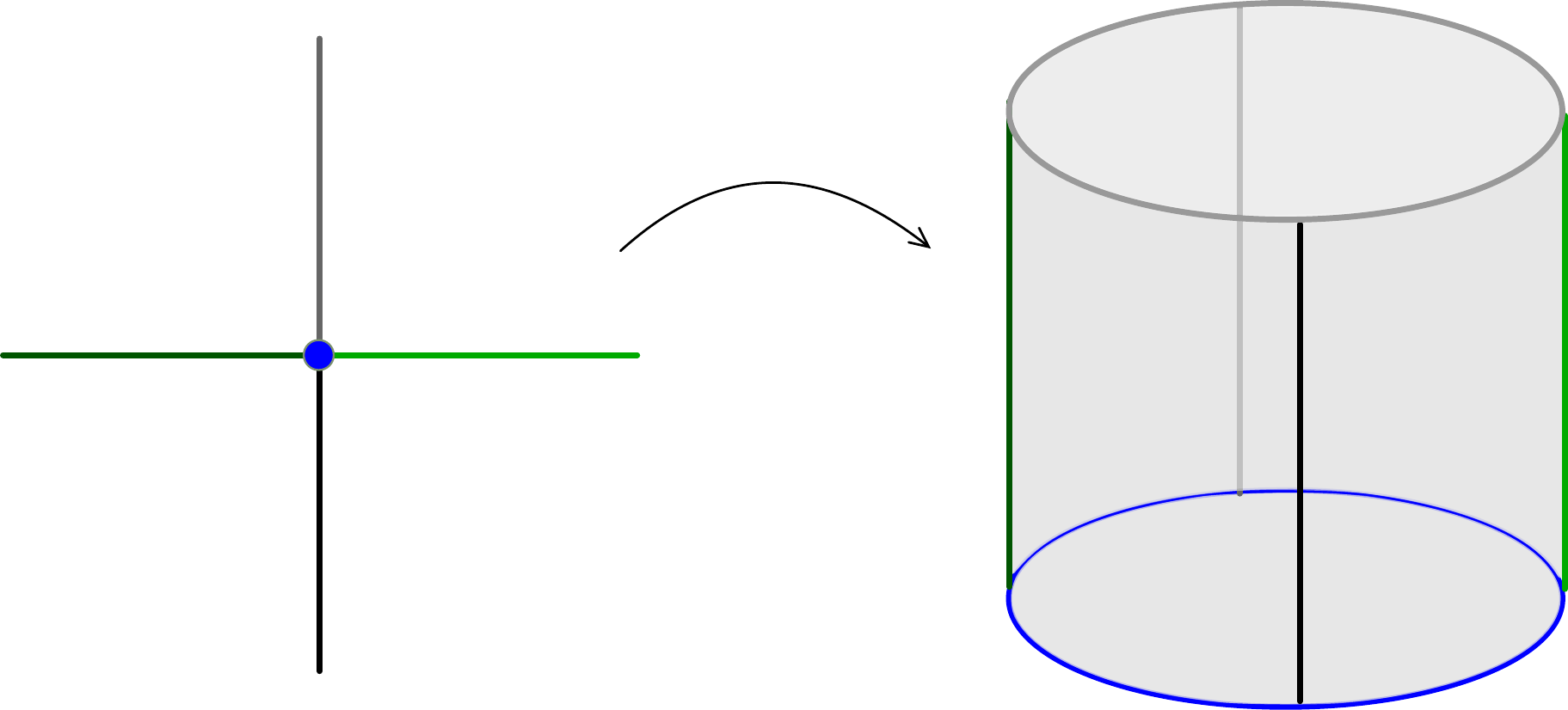}
		\end{overpic}
		\bigskip
		\caption{Blow-up of $p$ into $S^{1}$.}	
	\end{figure} 
	
	Let $Z$ be a small perturbation of $Z_{0}$, take it small enough such that the normal parameters $(\tilde{\ag},\tilde{\bg},\tilde{\cg})$ of $Z$ are close enough to $(\ag,\bg,\cg)$.
	
	If $\phi$ is the first return map associated to $Z$ at the fold-fold point $q(Z)\approx p$, then it has eigenvalues $\tilde{\lambda}_{\pm}=\tilde{a}\pm i\tilde{b}$. 
	
	Applying the same process to $\phi$, we can blow-up its singularity $q(Z)$ into $S^{1}$, and the dynamics in $S^{1}$ is induced by $\overline{\psi}:S^{1}\rightarrow S^{1}$, given by $\overline{\psi}(\theta)=\theta +\tilde{\tau},$	where $\tilde{a}+i\tilde{b}=e^{i\tilde{\tau}}$.

	Now, if $h: V(p)\rightarrow V(q(Z))$ is an equivalence between $Z_{0}$ and $Z$, then $h(S_{X_{0}})=S_{X}$. In adequate coordinates, it means that $h(x,0)=(f(x),0)$, where $f$ is a homeomorphism of the real line such that $f(0)=0$. 
	
	Notice that the motion of $S_{X_{0}}\cap\{x\geq0\}$ (resp. $S_{X}\cap\{x\geq0\}$) around the origin through $\phi_{0}$ (resp. $\phi$) is given by the orbit $\gamma_{0}=\{\overline{\psi_{0}}^{n}(0), n\in\Z\}$ (resp. $\gamma=\{\overline{\psi}^{n}(0), n\in\Z\}$).
	
	Since $h$ is an equivalence, it follows that the orbits $\gamma_{0}$ and $\gamma$ have the same topology. Nevertheless, if $\tau\in \Q$ (resp. $\tau\notin \Q$) we can take $Z$ (sufficiently near of $Z_{0}$) such that $\tilde{\tau}\notin \Q$ (resp. $\tilde{\tau}\in \Q$). Therefore, $\gamma_{0}$ is a periodic orbit and $\gamma$ is dense in $S^{1}$ (resp. $\gamma_{0}$ is dense in $S^{1}$ and $\gamma$ is a periodic orbit).
	
	It means that, when $\tau\in \Q$ (and $\gamma_{0}$ is periodic), the curves $\phi^{n}(S_{X})$ are tangent to a finite number of directions at $p$, i.e., there exists $m$ vectors $v_{1},\cdots, v_{m}$ in $T_{p}\s$ such that $ T_{p}\phi^{n}(S_{X})=\textrm{span}\{v_{i(n)}\}$, for some $i(n)\in\{1,\cdots, m\}$, for each $n\in \N$. Hence, we conclude that $\bigcup \phi^{n}(S_{X})$ has zero measure in $\s$.
	
	On the other hand, if $\tau\notin \Q$ (and $\gamma_{0}$ is dense), we have that for each $v\in T_{p}\s$, there exists a sequence $\phi^{n_{k}}(S_{X})$, such that $T_{p}\phi^{n_{k}}(S_{X})=\textrm{span}\{v_{k}\}$, and $v_{k}\rightarrow v$ when $k\rightarrow\infty$. We conclude that $\bigcup \phi^{n}(S_{X})$ has full measure in $\s$.

	From these facts, we can see that the orbits $\phi_{0}^{n}(S_{X_{0}})$ and $\phi^{n}(S_{X})$ do not have the same topology.
	
	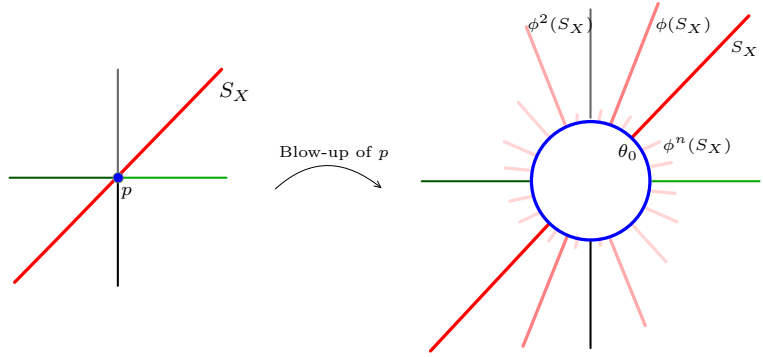
\begin{figure}[H]
		\centering
		\bigskip
		\begin{overpic}[width=10cm]{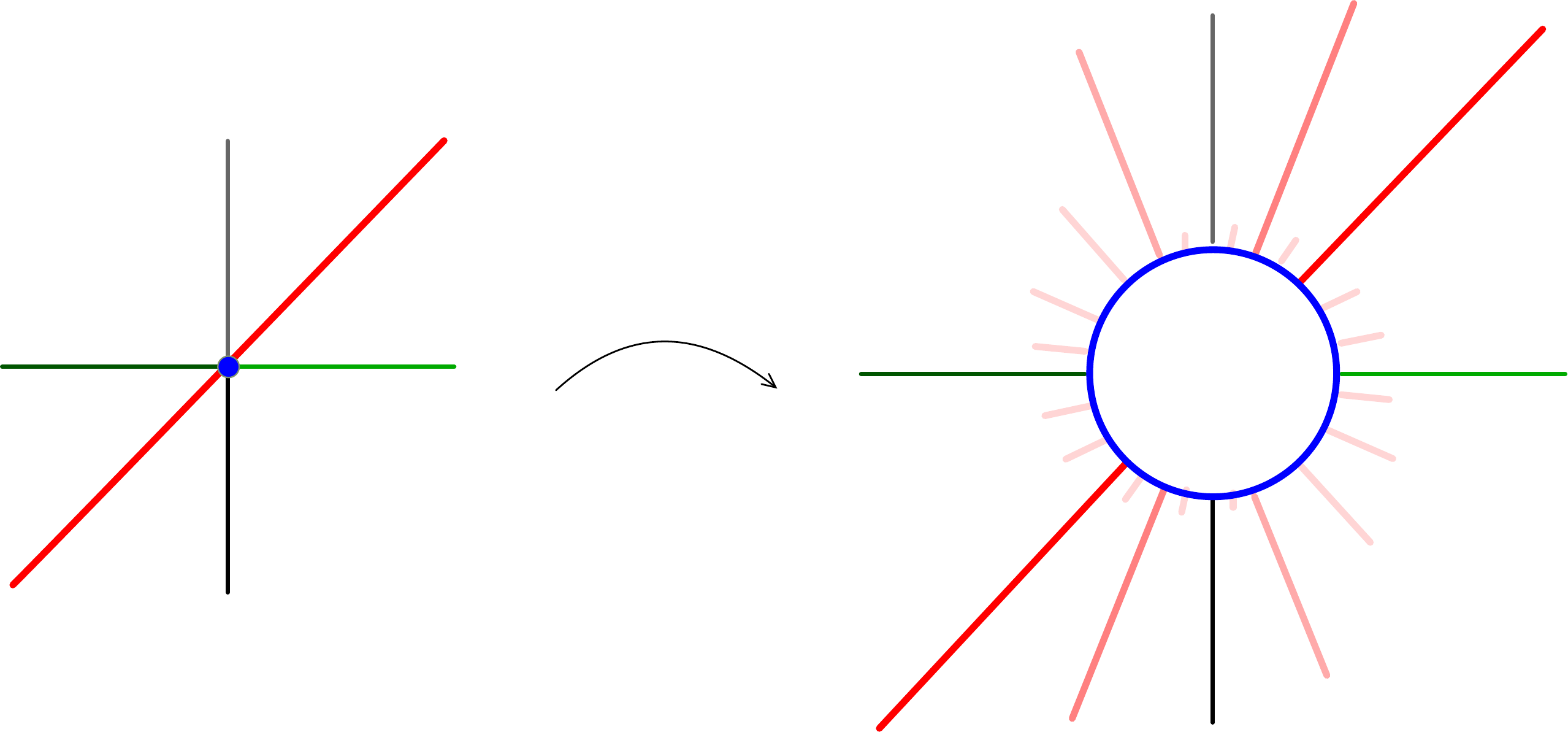}
			
			\put(96,40){{\tiny $S_{X}$}}
			\put(86,43){{\tiny $\phi(S_{X})$}}	
			\put(69,43){{\tiny $\phi^{2}(S_{X})$}}
			\put(87,27){{\tiny $\phi^{n}(S_{X})$}}											
			\put(28,34){{\footnotesize $S_{X}$}}
			\put(15,21){{\scriptsize $p$}}	
			\put(81,26){{\scriptsize $\theta_{0}$}}				
			\put(36,26){{\tiny Blow-up of $p$}}																					
		\end{overpic}
		\bigskip
		\caption{Behavior of $S_{X}$ when $\theta\notin\mathbb{Q}$.}	
	\end{figure} 
	
	Now, a $\s$-equivalence between $Z_{0}$ and $Z$ has to satisfy $h(S_{X_{0}})=S_{X}$ and $h\circ \phi_{0}=\phi\circ h$. Since $\phi_{0}^{n}(S_{X_{0}})$ and $\phi^{n}(S_{X})$ have different topological type, it follows that there is no $\s$-equivalence between $Z_{0}$
	and $Z$.
	
	We conclude that, in any neighborhood of $Z_{0}$ in $\Or$ we can find a nonsmooth vector field $Z$ such that $Z_{0}$ is not topologically equivalent to $Z$ at $p$. Therefore, $Z_{0}$ is locally structurally unstable at $p$.
\end{proof}

\begin{remark}
	Let $\tau_{Z}$ be the argument of the eigenvalues $a\pm ib $ of the first return map $\phi$ associated to $Z$.
	
	If $Z_{0}$ is a nonsmooth vector field satisfying the hypotheses of Proposition \ref{NH-unst}, then a neighborhood $\V_{0}$ of $Z_{0}$ in $\Or$ is foliated by codimension one submanifolds of $\Or$ corresponding to the value of $\tau_{Z}$, i.e., $Z_{1}\in \V_{0}$ and $Z_{2}\in \V_{0}$ lies on the same leaf if and only if $\tau_{Z_{1}}=\tau_{Z_{2}}$.
	
	The topological type of the first return map is locally constant along each leaf. Moreover, if $Z_{1}$ and $Z_{2}$ are elements of $\V_{0}$ lying on different leaves of the foliation then they are not topologically equivalent.
	
	We conclude that $Z_{0}$ has $\infty$ moduli of stability. (See \cite{BT,M,P} for more details.)
\end{remark}

Now we can prove Theorem D.

\begin{theorem}
	$\s_{0}$ is not residual in $\Or$.
\end{theorem}
\begin{proof}[Proof of Theorem D]
	It follows directly from Theorem \ref{NH-unst}. In fact, let $Z_{0}\in \Or$ and let  $(\ag_{0},\bg_{0},\cg_{0})$ be the normal parameters of $Z_{0}$ at $p$, they satisfy $\ag_{0}\bg_{0}(\ag_{0}\bg_{0}-\cg_{0})<0.$
	
	From continuity (and Implicit Function Theorem), there exist neighborhoods $\V$ of $Z_{0}$ in $\Or$ and $V$ of $p$ in $M$ such that, each $Z$ has a T-singularity at $q(Z)\in V$. 
	
	Moreover, if we apply Proposition \ref{FFNF_prop} to $Z$ at $q(Z)$, the normal parameters $(\ag,\bg,\cg)$ of  $Z$ at $q(Z)$ also satisfy $\ag\bg(\ag\bg-\cg)< 0.$
	
	
	From Theorem \ref{NH-unst}, each $Z\in \V$ is locally structurally unstable at the fold-fold singularity $q(Z)\in V\cap \s$. It means that each $Z\in\V$ is locally structurally unstable at a point $q(Z)\in\s$, hence each $Z\in\V$ is $\s$-locally structurally unstable. 
	Hence, $\V\subset \Or\setminus \s_{0}$ and $\s_{0}$ is not residual in $\Or$.		
\end{proof}

Notice that the results obtained until this point are mainly concerned with the foliation $\mathcal{F}$ generated by a nonsmooth vector field near a T-singularity. The sliding dynamics does not have influence on these results. Nevertheless, the existence of sliding vector fields will be extremely important in the classification of the structural stability of a T-singularity having a first return map with hyperbolic fixed point.

\begin{proposition}\label{webproof}
	Let $Z_{0}=(X_{0},Y_{0})\in\Or$ be a germ of nonsmooth vector field having a T-singularity at $p$. Let $(\ag,\bg,\cg)$ be the normal parameters of $Z_{0}$ at $p$. If either $\ag\bg\geq\cg$ and $\ag,\bg>0$  or $\ag\bg<0$, then $Z_{0}$ is locally structurally unstable at $p$.	
\end{proposition}
\begin{proof}
	%
	%
	%
	
	In the conditions of the theorem, we can use Lemma \ref{posinv} to conclude that the first return map $\phi_{0}$ of $Z_{0}$ has a local invariant manifold of the saddle contained in $\s^{s}$.
	
	Without loss of generality, assume that $W^{s}\subset \s^{s}$. Notice that the map $\phi_{0}^{2}$ has the same invariant manifolds of $\phi_{0}$, but it has both positive eigenvalues $0<\la<1<\mu$.
	
	Generically, we have that the sliding vector field $F_{0}$ is transverse to $W^{s}\cap \s^{ss}$ for a small neighborhood of $p$. Let $V=U\cap \s^{s}$, where $U$ is a neighborhood of $p$ such that $F_{0}$ is transverse to $W^{s}\cap V$.
	
	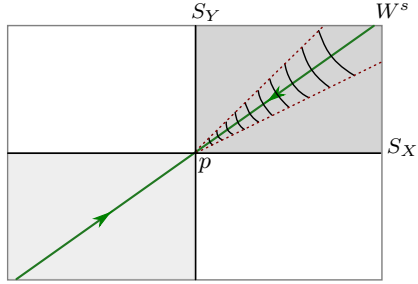
\begin{figure}[H]
		\centering
		\begin{overpic}[width=5cm]{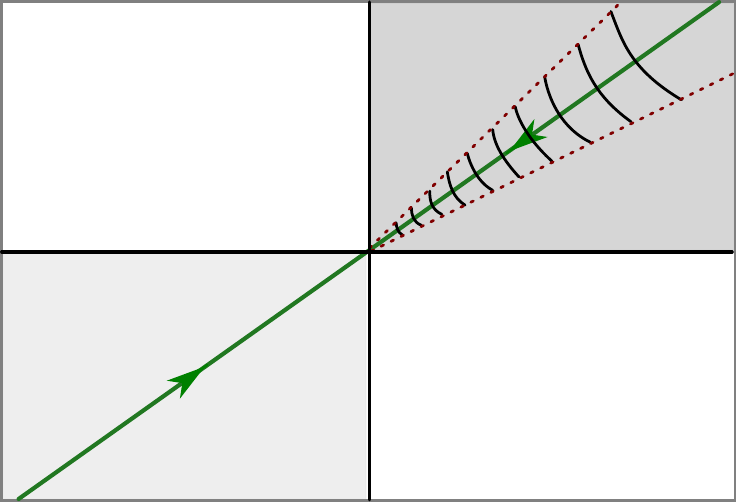}
			\put(51,30){{\small $p$}}
			\put(49,70){{\footnotesize $S_{Y}$}}		
			\put(101,34){{\footnotesize $S_{X}$}}			
			\put(98,70){{\footnotesize $W^{s}$}}										
		\end{overpic}
		\caption{Vector field $F_{0}$ near $W^{s}$.}
	\end{figure}
	
	Since $\la>0$, we have that  $\phi_{0}^{2}(W^{s})\subset \phi_{0}^{2}(V)\cap V.$ Moreover, $$\phi_{0}^{2n}(W^{s})\subset \phi_{0}^{2n}(V)\cap \phi_{0}^{2(n-1)}(V)\cap \cdots\cap \phi_{0}^{2}(V)\cap V,$$
	for each $n\in \N$.
	
	Let $R_{n}$ be the open set 
	$\phi_{0}^{2n}(V)\cap \phi_{0}^{2(n-1)}(V)\cap \cdots\cap \phi_{0}^{2}(V)\cap V$. Notice that, in each region $\phi_{0}^{2i}(V)$, we have a (push-forwarded) vector field $$F_{i}=(\phi_{0}^{2i})^{*}(F_{0}),$$ defined on it. Therefore, there are $n+1$ vector fields defined on $R_{n}$. Moreover, we can reduce $R_{n}$ such that $F_{i}$ and $F_{j}$ are transversal at each point of $R_{n}$, for $i\neq j$, generically. In fact, consider the expressions of $\phi_{X}$, $\phi_{Y}$ and $F_{Z}^{N}$ in the normal coordinates. Consider the curves $\gamma_{\pm}(t)=tv_{\pm}$, where $v_{\pm}$ are the eigenvectors associated to the eigenvalues $\la_{\pm}$ of $d\phi_{0}^{2}$. A simple computation shows that:
	$$F_{ij}^{\pm}(t)=\det(F_{i}(\gamma_{\pm}(t)),F_{j}(\gamma_{\pm}(t)))= A_{ij}^{\pm}(\ag,\bg,\cg)t^{2}+\er(t^{3}),$$
	where $A_{ij}^{\pm}$ is a rational function depending on $\ag,\bg$ and $\cg$.
	
	Clearly, if $A_{ij}^{\pm}\neq 0$, then $F_{i}$ and $F_{j}$ are transversal in a neighborhood of $\gamma_{\pm}$. In particular, they are transversal in a neighborhood of $W^{s}$.
	
	Since $A_{ij}^{\pm}=0$, for each $i,j=0,1,2$, defines a zero measure set in the parameter space $(\ag,\bg,\cg)$, we achieved our goal.
	
	Notice that, each vector field $F_{i}$ in $R_{n}$ defines a codimension one foliation $\mathcal{F}_{i}$ of $R_{n}$ ($R_{n}$ is foliated by the integral curves of the vector field $F_{i}$). Moreover, $(\mathcal{F}_{0},\cdots, \mathcal{F}_{n})$ is in general position (by the reduction of $R_{n}$). In particular, for $n=2$, we obtain $3$ foliations $(\mathcal{F}_{0},\mathcal{F}_{1},\mathcal{F}_{2})$ of $R_{2}$. This is called a \textbf{$3$-web} in $R_{2}$ (see \cite{B} and \cite{WEB}).
	
	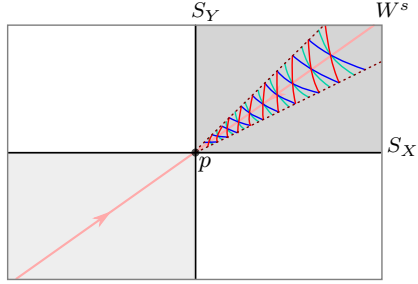
\begin{figure}[H]
		\centering
		\begin{overpic}[width=5cm]{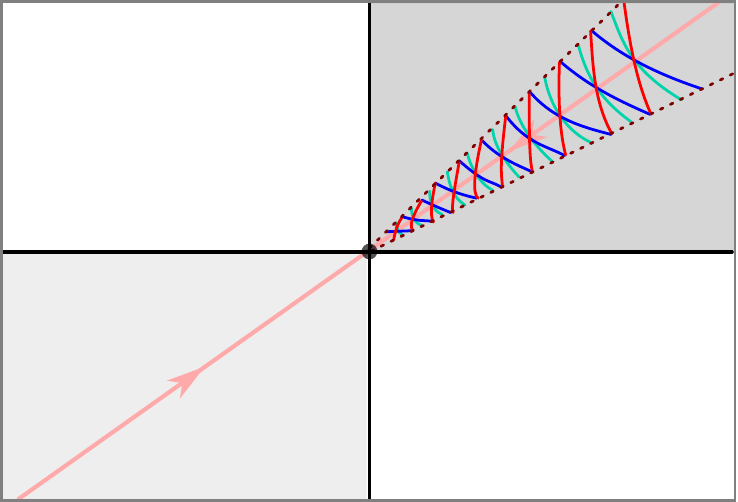}
			\put(51,30){{\small $p$}}
			\put(49,70){{\footnotesize $S_{Y}$}}		
			\put(101,34){{\footnotesize $S_{X}$}}			
			\put(98,70){{\footnotesize $W^{s}$}}										
		\end{overpic}
		\caption{Foliations $\mathcal{F}_{0},\mathcal{F}_{1}$, and $\mathcal{F}_{2}$ originated from the vector fields $F_{0}$, $F_{1}$ and $F_{2}$, respectively, near $W^{s}$.}
	\end{figure}
	
	Since $R_{2}$ is  a $2$-dimensional manifold, it follows that these foliations are structurally unstable in the following sense. If $(\widetilde{\mathcal{F}_{0}},\widetilde{\mathcal{F}_{1}},\widetilde{\mathcal{F}_{2}})$ are the foliations correspondent to a nonsmooth vector field $\widetilde{Z}\approx Z_{0}$, then there exists at least one $\widetilde{Z}$ such that there is no homeomorphism $h:R_{2}\rightarrow \widetilde{R_{2}}$ satisfying $h(\mathcal{F}_{i})=\widetilde{\mathcal{F}_{i}}$, for every $i=0,1,2$, preserving the leaves of each foliation.
	
	Clearly the property above has to be preserved by a $\s$-equivalence, hence there exists a $Z$ sufficiently near of $Z_{0}$ which is topologically different from $Z_{0}$ near $p$.
	
	The instability of $Z_{0}$ at $p$ follows directly from these facts.
\end{proof}

\begin{remark}
	In general, the Theory of Webs used in the last Theorem is developed for foliations on $\mathbb{C}^{n}$. Nevertheless, we can identify $\s$ with $\mathbb{C}$ at $p$ (since $\s$ is $2$-dimensional) and apply the results of this theory for this case.
\end{remark}

%
%
%
%
%

Now, let $Z_{0}=(X_{0},Y_{0})\in\Or$ be a germ of nonsmooth vector field having a Teixeira singularity at $p$. Let $(\ag,\bg,\cg)$ be the normal parameters of $Z_{0}$ at $p$ and assume that $\ag\bg\geq\cg$ and $\ag,\bg<0$.

Let $Z\in \Or$ be any small perturbation of $Z_{0}$ and denote their first return maps by $\phi$ and $\phi_{0}$, respectively. Our goal is to construct a topological equivalence between $Z$ and $Z_{0}$.

Using the Implicit Function Theorem and the continuous dependence between $Z_{0}$ and its normal parameters, we can deduce the following result:

\begin{lemma} \label{lemadasela}
	There exists a neighborhood $\V$ of $Z_{0}$ such that, for each $Z\in\V$, $F_{Z}^{N}$ and  $F_{Z_{0}}^{N}$ have the same topological type and the first return map $\phi$ of $Z$ has a saddle at the origin with both local invariant manifolds in $\s^{c}$.
\end{lemma}

\begin{remark}
	In what follows, $\V$ will denote the neighborhood of Lemma \ref{lemadasela}.
\end{remark}

Now we prove the existence of an invariant nonsmooth diabolo in an analytic way, this result was achieved by M. Jeffrey and A. Colombo for the semi-linear case (see \cite{J1}).

\begin{proposition}\label{diabolo-prop}
	Let $Z_{0}=(X_{0},Y_{0})\in\Or$ be a nonsmooth vector field having a T-singularity at $p$ such that the normal parameters $(\ag,\bg,\cg)$ of $Z_{0}$ at $p$ satisfy $\ag\bg\geq\cg$ and $\ag,\bg<0$. Then $Z_{0}$ has a invariant nonsmooth diabolo $D_{0}$ which prevents connections between points of $\s^{us}$ and $\s^{ss}$  through orbits of $Z$.  
\end{proposition}
\begin{proof}
	From Lemma \ref{lemadasela}, it follows that the first return map $\phi_{0}=\phi_{X_{0}}\circ\phi_{Y_{0}}$ associated to $Z_{0}$ has a hyperbolic saddle at $p$ with both eigenvectors in $\s^{c}$.
	
	Notice that the local stable manifold of the saddle $W^{s}$ is tangent to the eigenvector $v_{-}$ correspondent to the eigenvalue $\la$ and the local unstable manifold of the saddle $W^{u}$ is tangent to the eigenvector $v_{-}$ correspondent to the eigenvalue $\mu$, where $|\la|<1<|\mu|$.
	
	Moreover, $W^{s}$ and $W^{u}$ are curves on $\s$ passing through $p$ transverse to $S_{X}\cup S_{Y}$ at $p$ and $W^{s}\pitchfork W^{u}$ at $p$ ($p$ is hyperbolic). Using coordinates $(x,y)$ at $p$ (which put $Z_{0}$ in the normal form \ref{FFNF}), we can see that, $S_{X_{0}}= \fix(\phi_{X_{0}})$ is the $x$-axis, $S_{Y_{0}}= \fix(\phi_{Y_{0}})$ is a curve tangent to $y$-axis at $0$, and $W^{s}$ and $W^{u}$ are curves passing through $0$ contained in the second and the fourth quadrants which are transverse to $S_{X_{0}}\cup S_{Y_{0}}$ at $0$.
	
	Therefore we have the following situation:
	
	\begin{figure}[H]
		\centering
		\bigskip
		
		\begin{overpic}[width=6cm]{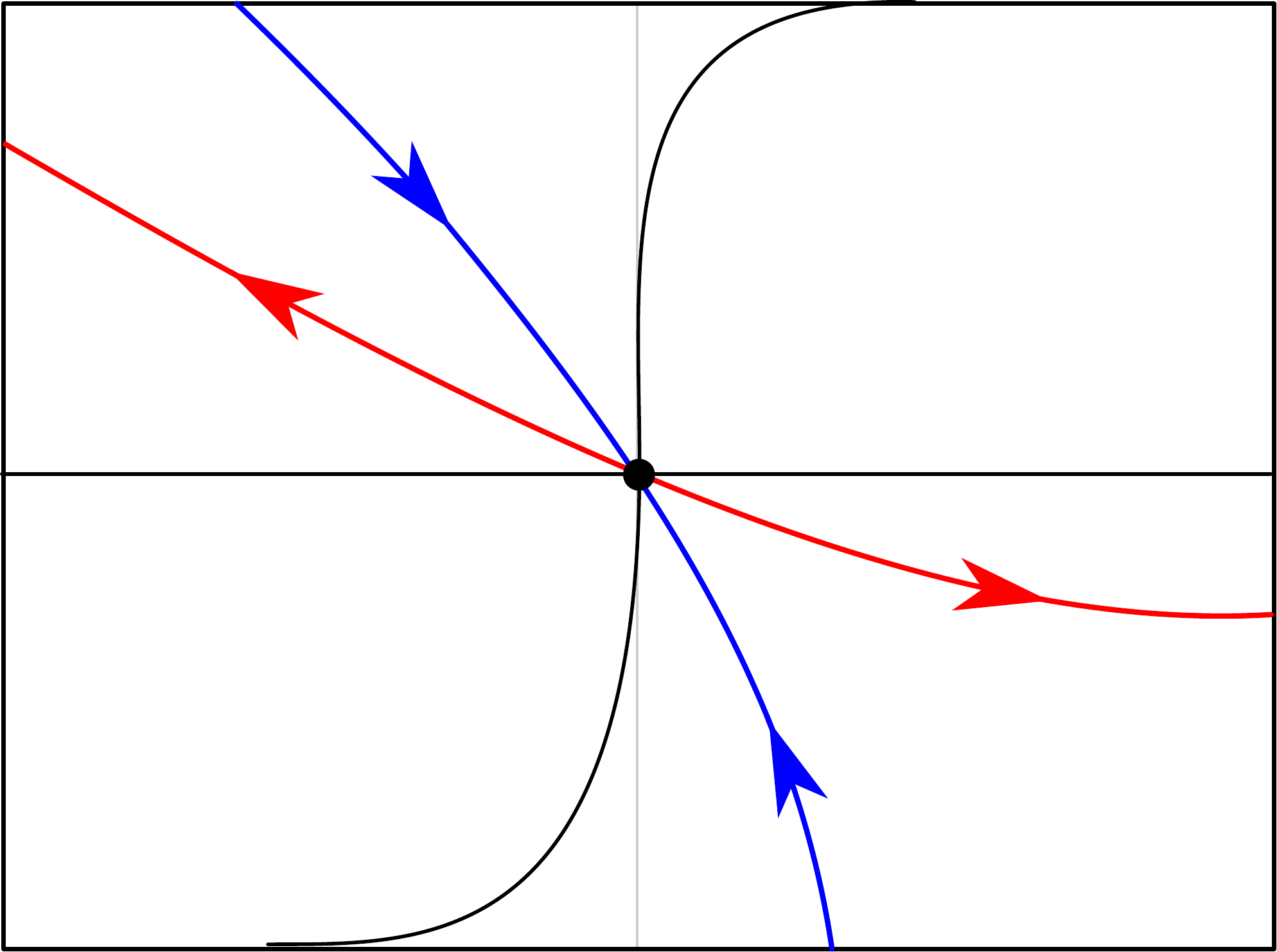}
			\put(65,70){{\scriptsize $S_{Y_{0}}$}}		
			\put(90,39){{\scriptsize $S_{X_{0}}$}}		
			\put(24,70){{\scriptsize $W^{s}$}}		
			\put(1,64){{\scriptsize $W^{u}$}}							
		\end{overpic}
		\bigskip
	\end{figure} 
	
	Now, from Proposition \ref{reversibility}, it follows that $\phi_{X_{0}}(W^{u})\subset W^{s}$, but the image of a point in the semi-plane $\{y>0\}$ through $\phi_{X_{0}}$ is a point in the semi-plane $\{y<0\}$ by the construction of $\phi_{X_{0}}$. It means that the branch of $W^{u}$ in the second quadrant has to be taken into the branch of $W^{s}$ in the fourth quadrant.
	
	Also,  $\phi_{Y_{0}}(W^{s})\subset W^{u}$. Notice that, $S_{Y_{0}}$ splits $\R^{2}$ in two connected components, $C_{-}$ and $C_{+}$. From the construction of $\phi_{Y_{0}}$, the image of a point in $C_{-}$ through $\phi_{Y_{0}}$ is a point in $C_{+}$.  It means that the branch of $W^{s}$ in the fourth quadrant is taken into the branch of $W^{u}$ in the second quadrant.
	
	These connections produce an invariant (nonsmooth) cone with vertex at the fold-fold point which contains $\s^{us}$ in its interior. Analogously, we prove that there exists an invariant (nonsmooth) cone with vertex at the fold-fold point which contains $\s^{ss}$ in its interior. These two cones produce the required nonsmooth diabolo (see Figure \ref{diabolo}).
\end{proof}

\begin{remark}
	In another words, there is no communication between $\s^{us}$ and $\s^{ss}$ in this case.
\end{remark}

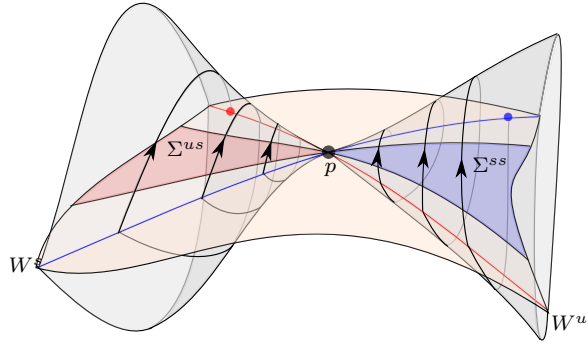
\begin{figure}[H]
	\centering
	\bigskip
	\begin{overpic}[width=7cm]{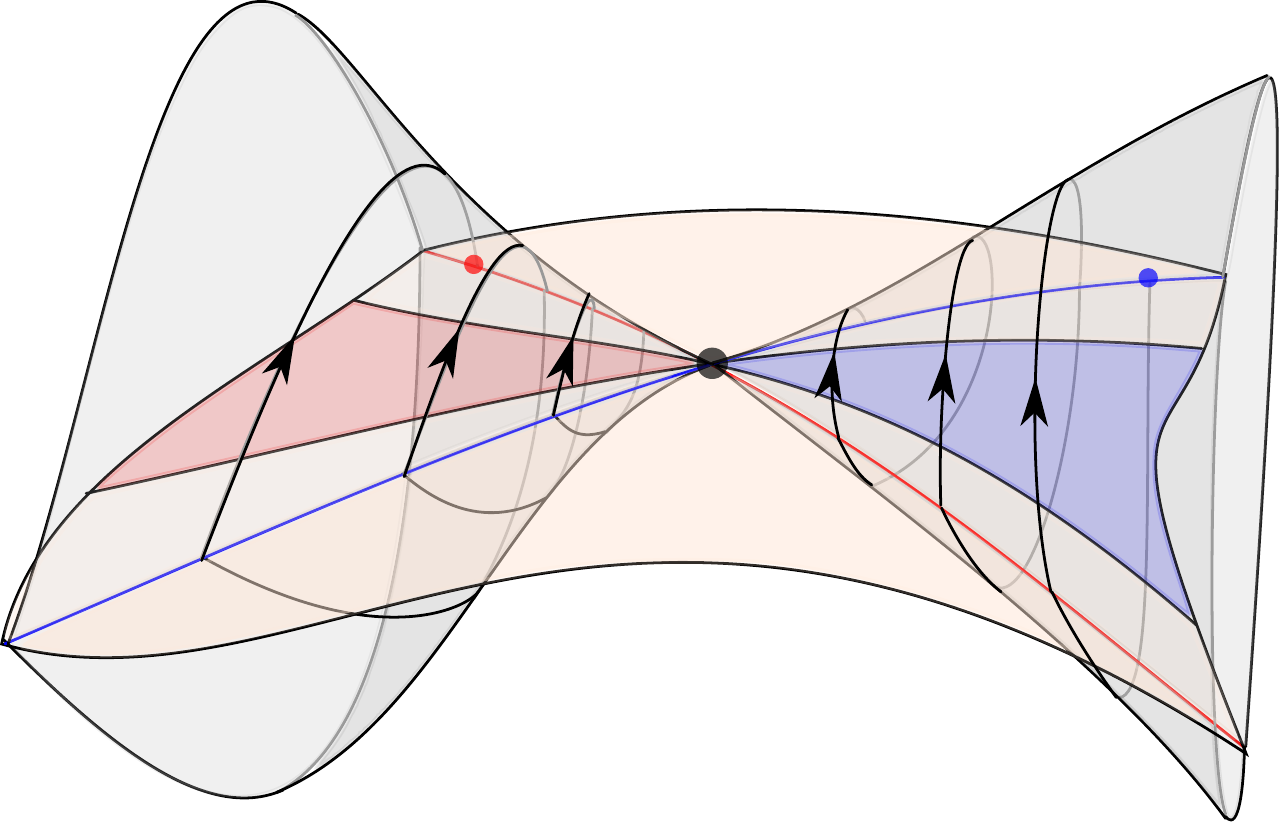}
		\put(55,32){{\footnotesize $p$}}		
		\put(25,35){{\footnotesize $\s^{us}$}}		
		\put(83,32){{\footnotesize $\s^{ss}$}}			
		\put(98,2){{\footnotesize $W^{u}$}}		
		\put(-5,13){{\footnotesize $W^{s}$}}				
	\end{overpic}
	\bigskip
	\caption{A nonsmooth diabolo $D_{0}$ of $Z_{0}$.}
	\label{diabolo}	
\end{figure} 

Now we proceed by constructing a homeomorphism between $Z\in\V$ and $Z_{0}$. 

\begin{lemma}\label{fund}
	If $Z\in \V$, there exists an order-preserving homeomorphism $h:\s^{s}(Z_{0})\rightarrow \s^{s}(Z)$ which carries orbits of $F_{Z_{0}}$ onto orbits of $F_{Z}$.
\end{lemma}
The proof of this lemma follows straightforward from Lemmas \ref{lemadasela} and \ref{sliding}.

\begin{definition}
	If $\phi:(\R^{2},0)\rightarrow (\R^{2},0)$ is a germ of diffeomorphism at $0$ having a saddle at $0$, then the \textbf{deMelo-Palis invariant} of $\phi$ is defined as:
	$$P(\phi)=\frac{\log(|\la|)}{\log(|\mu|)},$$
	where $\la,\mu$ are the eigenvalues of $d\phi(0)$ and $|\la|<1<|\mu|$.
\end{definition}

\begin{remark}
	In fact, the deMelo-Palis invariant $P$ is a moduli of stability for $\phi$. (See \cite{M,P}.)
\end{remark}

\begin{proposition} \label{sela-prop}
	If $Z\in\V$, there exists a homeomorphism $h:\s\rightarrow\s$ which is a continuous extension of the homeomorphism $h:\s^{s}(Z_{0})\rightarrow \s^{s}(Z)$ given by Lemma \ref{fund}, such that $\phi\circ h=h\circ \phi_{0}$, i.e. it is a topological equivalence between $\phi$ and $\phi_{0}$.
\end{proposition}
\begin{proof}
	The proof of this proposition is divided into steps.
	
	Let $h:\s^{s}(Z_{0})\rightarrow \s^{s}(Z)$ be the homeomorphism obtained in Lemma \ref{fund}.
	
	Notice that $Z$ has a T-singularity at $q(Z)\approx p$. Since $F_{Z_{0}}^{N}$ and $F_{Z}^{N}$ are transversal to $S_{Z_{0}}\setminus \{p\}$ and $S_{Z}\setminus\{q(Z)\}$, respectively, we can easily continuously extend $h$ on $\overline{\s^{s}(Z_{0})}$:
	$$h:\overline{\s^{s}(Z_{0})}\rightarrow\overline{\s^{s}(Z)},$$
	via limit.
	
	\textbf{Step 1:} The first task is to define a \textbf{fundamental domain} for the first return maps, $\phi$ and $\phi_{0}$.
	
	We will detail it for $\phi_{0}$. The process to construct the fundamental domain of $\phi$ is completely analogous.
	
	By the Linearization Theorem (see \cite{H}), we may assume that $\phi_{0}$ is linear. Moreover, we can consider coordinates $(x,y)$ of $\s$ at $p$ such that:
	$$\phi_{0}(x,y)=(\la_{0}x, \mu_{0}y),$$
	where $\la_{0},\mu_{0}$ are the eigenvalues of $\phi_{0}$ such that $|\mu_{0}|<1<|\la_{0}|$.	
	
	By the position of $S_{X_{0}}$, $S_{Y_{0}}$ and the invariant manifolds of the saddle, obtained in Proposition \ref{diabolo-prop}, it follows that:
	\begin{itemize}
		\item $S_{X_{0}}$ is a curve passing through $0$, with one branch in the first quadrant and another in the fourth;
		\item $S_{Y_{0}}$ is a curve passing through $0$, with one branch in the first quadrant and another in the fourth;		
		\item $S_{X_{0}}$ is tangent to the line $y=k_{0}x$;
		\item $S_{Y_{0}}$ is tangent to the line $y=K_{0}x$;
		\item $0<k_{0}<K_{0}.$
	\end{itemize}
	
	We have the following situation:
	
	\begin{figure}[H]
		\centering
		\bigskip
		\begin{overpic}[width=9cm]{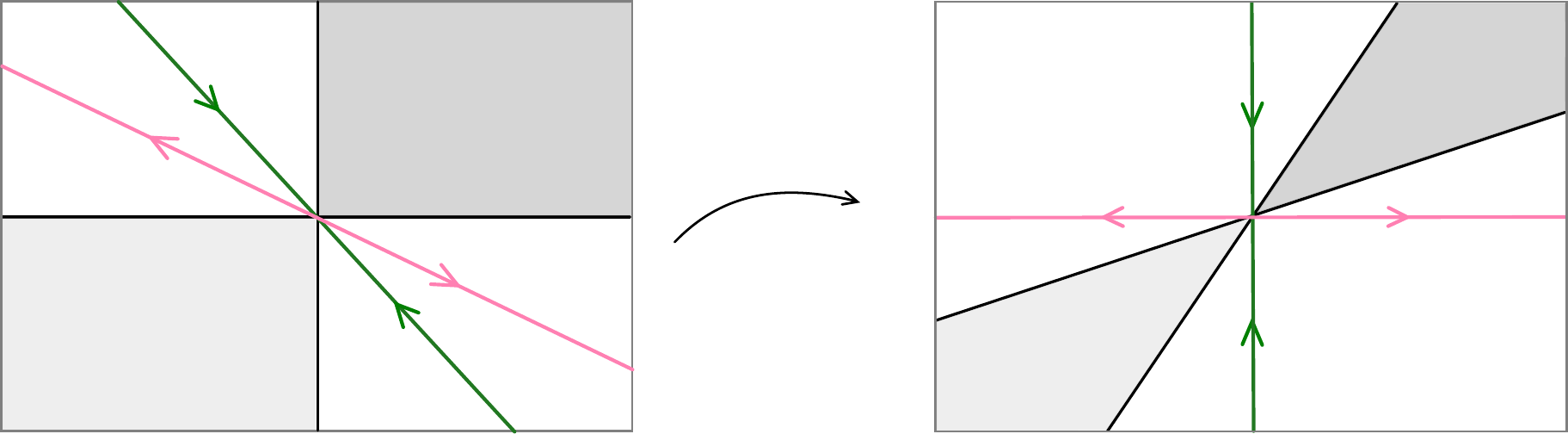}	
			\put(33,-3){{\footnotesize $W^{s}$}}
			\put(41,3){{\footnotesize $W^{u}$}}	
			\put(40,14){{\footnotesize $S_{X}$}}	
			\put(19,28){{\footnotesize $S_{Y}$}}
			\put(30,20){{\footnotesize $\s^{ss}$}}	
			\put(10,5){{\footnotesize $\s^{us}$}}	
			\put(80,-3){{\footnotesize $W^{s}$}}
			\put(102,13){{\footnotesize $W^{u}$}}	
			\put(102,20){{\footnotesize $S_{X}$}}	
			\put(87,28){{\footnotesize $S_{Y}$}}
			\put(90,20){{\footnotesize $\s^{ss}$}}	
			\put(67,5){{\footnotesize $\s^{us}$}}																						
		\end{overpic}
		\bigskip
		\caption{Change of coordinates.}
	\end{figure}
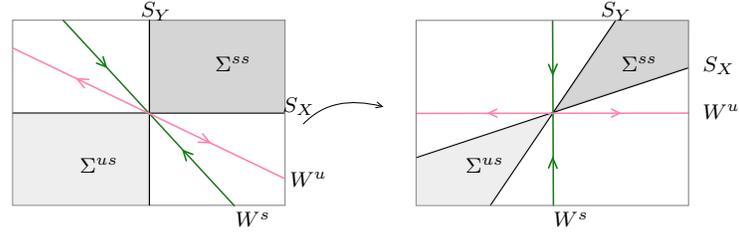 
	
	Without loss of generality, consider that $S_{X_{0}}=\{y=k_{0}x\}$ and $S_{Y_{0}}=\{y=K_{0}x\}$ and assume that these lines are the fixed points of $\phi_{X_{0}}$ and $\phi_{Y_{0}}$, respectively. It will reduce our work, nevertheless it generates no loss of generality, since the same can be done with the original sets.
	
	
	From the existence of the invariant diabolo in Proposition \ref{diabolo-prop}, it follows that, $\phi_{0}^{-1}(S_{X_{0}})$ is a line in the same region of $S_{X_{0}}$, moreover, its inclination is greater than $K_{0}$.
	
	Define:
	$$\omega_{0}=\{ (x,y);\ k_{0}x\leq y\leq K_{0}x\} \textrm{ and }\widetilde{\omega}_{0}=\phi_{Y_{0}}(\omega_{0}).$$
	
	Notice that $R_{0}=\omega_{0}\cup \widetilde{\omega}_{0}$ is the region delimited by the lines $S_{X_{0}}$ and $\phi_{0}^{-1}(S_{X_{0}})$.

	Now it is immediate that $\phi_{0}^{n}(S_{X_{0}})\rightarrow W^{u}$ when $n\rightarrow\infty$ and $\phi_{0}^{n}(S_{X_{0}})\rightarrow W^{s}$ when $n\rightarrow-\infty$. Therefore, the first and the third quadrants are partitioned by $\phi_{0}^{n}(R_{0})$, $n\in\Z$.
	
	In another words, if $Q=\{(x,y);\ xy>0\}$, then
	$$Q=\bigcup_{n\in\Z}\phi_{0}^{n}(R_{0}).$$
	Therefore, we say that $R_{0}$ is the fundamental domain of $\phi_{0}$.
	
	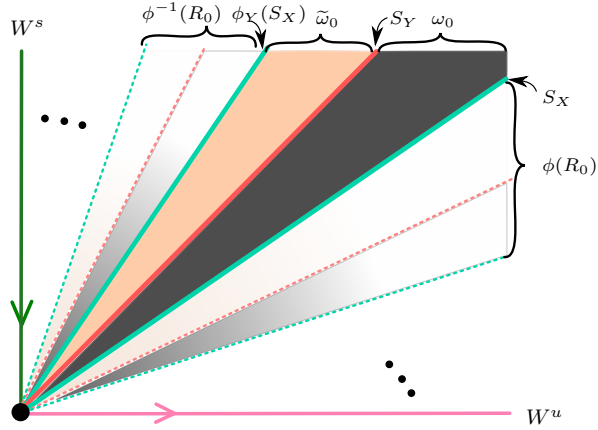
\begin{figure}[H]
		\centering
		\bigskip
		\begin{overpic}[width=7cm]{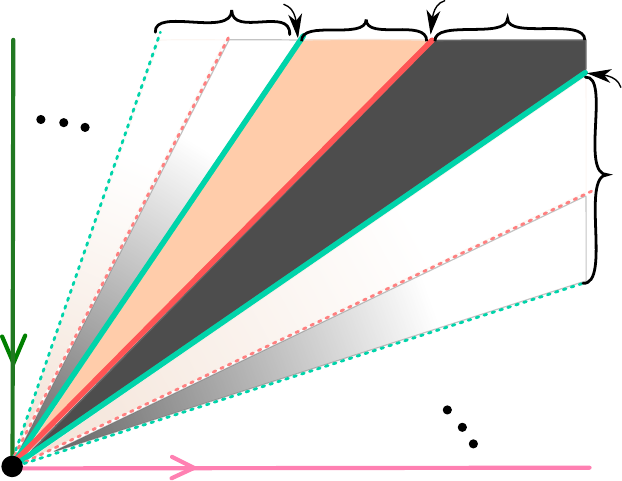}	
			\put(0,74){{\footnotesize $W^{s}$}}
			\put(98,0){{\footnotesize $W^{u}$}}	
			\put(25,77){{\scriptsize $\phi^{-1}(R_{0})$}}	
			\put(42,77){{\scriptsize $\phi_{Y}(S_{X})$}}
			\put(58,76){{\scriptsize $\widetilde{\omega}_{0}$}}	
			\put(72,76){{\scriptsize $S_{Y}$}}	
			\put(80,76){{\scriptsize $\omega_{0}$}}	
			\put(101,61){{\scriptsize $S_{X}$}}	
			\put(101,48){{\scriptsize $\phi(R_{0})$}}						
		\end{overpic}
		\bigskip
		\caption{Fundamental domain $R_{0}= \omega_{0}\cup\widetilde{\omega}_{0}$ in the first quadrant.}
	\end{figure} 
	
	Similarly, we can consider coordinates $(x,y)$ of $\s$ at $p$ such that:
	$$\phi(x,y)=(\la x, \mu y),$$
	where $\la,\mu$ are the eigenvalues of $\phi$ such that $|\mu|<1<|\la|$.
	Therefore, there exists $R=\omega\cup\widetilde{\omega}$, where $\omega$ is the region delimited by $S_{X}$ and $S_{Y}$ and $\widetilde{\omega}=\phi^{Y}(\omega)$.
	
	Also $Q=\bigcup_{n\in\Z}\phi^{n}(R)$, and $R$ is the region delimited by $S_{X}$ and $\phi^{-1}(S_{X})$.
	
	In both cases, each orbit of $\phi_{0}$ (and $\phi$) passes a unique time in each sector of the partition of $Q$.

	\textbf{Step 2:} Extending the domain of $h$ into $h:Q\rightarrow Q$.
	
	Notice that $h:\omega_{0}\rightarrow \omega$ is already defined (it is the homeomorphism $h:\overline{\s^{s}(Z_{0})}\rightarrow\overline{\s^{s}(Z)}$ in these coordinates).
	
	If $q\in \widetilde{\omega}_{0}$, then $q=\phi_{Y_{0}}(\tilde{q})$, for some $\tilde{q}\in\omega_{0}$, therefore, define:
	$$h(q)=\phi_{Y}(h(\tilde{q})).$$
	
	Clearly, it is a continuous extension of $h$ from $\omega_{0}$ into $R_{0}$. Now, we have defined a homeomorphism $h:R_{0}\rightarrow R$.
	
	The extension to $Q$ follows in a natural way (since it is defined in a fundamental domain).
	
	In fact, if $q\in Q$, there exists	a unique $\tilde{q}\in R_{0}$ and a unique $n\in\Z$, such that $q=\phi_{0}^{n}(\tilde{q})$. Define:
	$$h(q)=\phi^{n}(h(\tilde{q})).$$
	
	Clearly, $h: Q\rightarrow Q$ is a homeomorphism satisfying:
	$$h( \phi_{0}(q))=\phi(h(q)),$$
	for each $q\in Q$.
	
	\textbf{Step 3:} Extending $h$ on both $W^{u}$ and $W^{s}$ in a continuous fashion.
	
	This is the most delicate part of the proof.
	
	Consider an arbitrary continuous extension of $h$ on $W^{s}$.
	
	Now, the difficult task is to continuously extend it to $W^{u}$, and it will be only possible because:
	$$P(\phi_{0})=-1=P(\phi),$$
	where $P$ is the deMelo-Palis invariant. 
	
	Only the extension in the first quadrant will be detailed. The extensions in the other quadrants are similar. 
	
	We extend $\phi$ in the following way. 
	
	Fix $w=(d,0)\in W^{u}$, then, there exists a sequence $w_{i}=\phi_{0}^{N_{i}}(y_{i})$ such that $N_{i}\rightarrow\infty$ when $i\rightarrow\infty$ and $y_{i}$ is a sequence contained in $S_{X_{0}}\cap\{x,y>0 \}$ such that $y_{i}\rightarrow 0$ when $i\rightarrow\infty$, which satisfies:
	
	$$\lim_{i\rightarrow\infty}\phi_{0}^{N_{i}}(y_{i})=w.$$
	
	Notice that, the homeomorphism $h$ is already defined for the sequence $w_{i}$. Since we want a continuous extension and an equivalence, we must define:
	$$h(w)=\lim_{i\rightarrow\infty}h(\phi_{0}^{N_{i}}(y_{i}))=\lim_{i\rightarrow\infty}\phi^{N_{i}}(h(y_{i})).$$
	
	Our work is to prove that the limit above exists. Then, $h$ will be extended on $W^{u}$ by doing this process for every $q\in[w,\phi_{0}(w)]$ and then extend it through the images of this fundamental domain by $\phi_{0}$.		
	
	Now, we prove the existence of the limit.
	
	Since $\phi_{0}(S_{X_{0}})=S_{X}$ and $\phi^{n}(S_{X})\in W^{u}$, it follows directly that:
	$$\lim_{i\rightarrow\infty}\pi_{2}(\phi^{N_{i}}(h(y_{i})))=0.$$
	
	Therefore, $\pi_{2}(h(w))=0$ and it is well-defined. The problem happens for the first coordinate.
	Consider:
	\begin{enumerate}
		\item $w=(d,0)$;
		\item $y_{i}\rightarrow 0$, $y_{i}\in S_{X_{0}}$, for every $i$;
		\item $N_{i}\rightarrow$ such that $\phi_{0}^{N_{i}}(y_{i})=w_{i}\rightarrow w$;
		\item $t_{i}\rightarrow\infty$, $x_{i}\rightarrow x\in W^{s}$ such that $y_{i}=\phi_{0}^{t_{i}}(x_{i})$.
	\end{enumerate}
	
	\begin{figure}[H]
		\centering
		\bigskip
		\begin{overpic}[width=6cm]{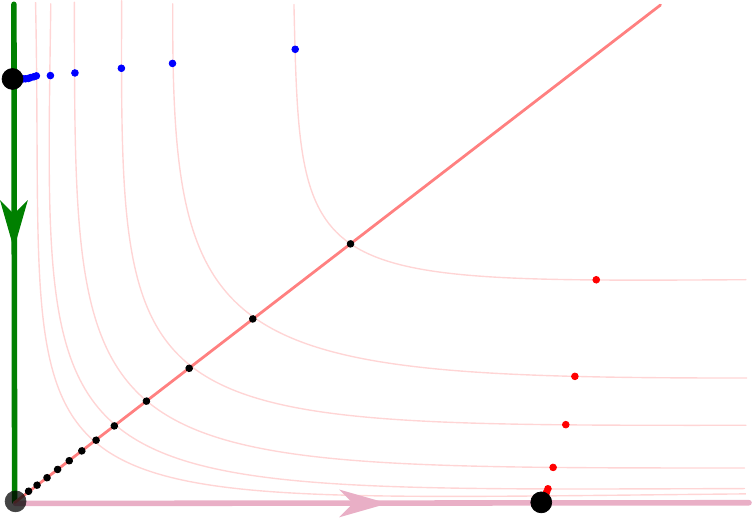}	
			\put(-5,57){{\footnotesize $x$}}
			\put(41,61.5){{\footnotesize $x_{i}$}}	
			\put(26,60){{\footnotesize $x_{i+1}$}}	
			\put(-4,-3.5){{\footnotesize $\vec{0}$}}
			\put(50,36){{\footnotesize $y_{i}$}}
			\put(37,25){{\footnotesize $y_{i+1}$}}		
			\put(81,32){{\footnotesize $w_{i}$}}
			\put(77,20){{\footnotesize $w_{i+1}$}}		
			\put(70.5,-5){{\footnotesize $w$}}								
			\put(102,0){{\footnotesize $W^{u}$}}
			\put(0,69){{\footnotesize $W^{s}$}}		
			\put(83,61.5){{\footnotesize $S_{X_{0}}$}}			
		\end{overpic}
		\bigskip
		\caption{Sequences $(x_{i})$, $(y_{i})$ and $(w_{i})$.}
	\end{figure}
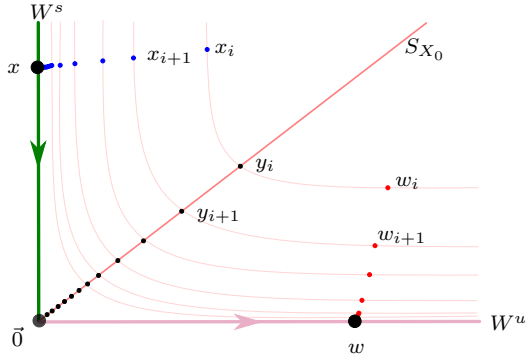 
	
	Now, denote: $d_{i}=\pi_{1}(w_{i})$, $\widetilde{d_{1}}=\pi_{1}(\widetilde{w_{i}})$, $\widetilde{y_{i}}=h(y_{i})$, $a_{i}=\pi_{2}(x_{i})$ and $\widetilde{a_{i}}=\pi_{2}(h(x_{i}))$. Hence, we must prove that $\widetilde{d_{i}}$ converges.
	
	Notice that, since $h$ is continuously extended for $W^{s}$, it follows that $\widetilde{a_{i}}$ is a convergent sequence. Let $\lim \widetilde{a_{i}}=\widetilde{a}$. 	
	$$
	\begin{array}{lclcl}
	\widetilde{d_{1}} &=&\vspace{0.3cm} \pi_{1}(h(\phi_{0}^{N_{i}}(y_{i}))) 
	%
	&=& \la^{N_{i}}\pi_{1}(\widetilde{y_{i}})
	\end{array}
	$$	
	Now, observe that:
	$$\widetilde{y_{i}}=h(y_{i})=h(\phi_{0}^{t_{i}}(x_{i}))=\phi^{t_{i}}(\widetilde{x_{i}})=(\la^{t_{i}}\pi_{1}(\widetilde{x_{i}}),\mu^{t_{i}}\pi_{2}(\widetilde{y_{i}})).$$
	
	Since $\tilde{y_{i}}\in S_{X}=\{y=kx\}$, it follows that:
	$$\pi_{1}(\widetilde{y_{i}})=\frac{1}{k}\pi_{2}(\widetilde{y_{i}})=\frac{1}{k}\mu^{t_{i}}\pi_{2}(\widetilde{y_{i}}).$$ 
	
	Hence:
	$$\widetilde{d_{i}}=\frac{1}{k}\la^{N_{i}}\mu^{t_{i}}\pi_{2}(\widetilde{y_{i}}),$$
	and applying the logarithm, we obtain:
	$$\log(\widetilde{d_{i}}k)=N_{i}\log(\la) +t_{i}\log(\mu)+\log(\widetilde{a_{i}}).$$
	
	With the same process, we also obtain:
	
	$$\log(d_{i}k_{0})=N_{i}\log(\la_{0}) +t_{i}\log(\mu_{0})+\log(a_{i}).$$
	
	Since $\log(d_{i}k_{0})$ and $\log(a_{i})$ converge, it follows that $N_{i}\log(\la_{0}) +t_{i}\log(\mu_{0})$ converges.
	
	Now, using that $P(\phi_{0})=P(\phi)$, it is immediate that $N_{i}\log(\la) +t_{i}\log(\mu)$ converges.
	
	Since $\widetilde{a_{i}}\rightarrow \widetilde{a}$, it follows that $\widetilde{d_{i}}$ converges and the proof is complete.
\end{proof}


\begin{remark}
	Notice that, both $\phi$ and $\phi_{0}$ are composition of elements of $W^{r}$, therefore a perturbation of the first return map $\phi_{0}$ still is a composition of two involutions. Hence the diffeomorphism $\phi_{0}$ is perturbed only over the codimension one submanifold $P^{-1}(-1)$ of Diff$(\rn{2},0)$ (space of germs of diffeomorphisms at $0$.).
\end{remark}
It follows straightforward from the previous results:
\begin{proposition}\label{proptsingstable}
	Let $Z_{0}=(X_{0},Y_{0})\in\Or$ be a germ of nonsmooth vector field having a Teixeira singularity at $p$. Let $(\ag,\bg,\cg)$ be the normal parameters of $Z_{0}$ at $p$. If $\ag\bg\geq\cg$ and $\ag,\bg<0$, then $Z_{0}$ is locally structurally stable at $p$.	
\end{proposition}	
%
%
%
%
%
%
%
%

Finally, we conclude the proof of Theorem A:

\begin{proof}[Proof of Theorem A]
	Notice that $Z$ satisfies condition $\s(E)$ at $p$ if, and only if, the normal parameters $(\ag,\bg,\cg)$ of $Z$ at $p$ satisfy $\ag\bg\geq\cg$ and $\ag,\bg<0$.
	
	The result follows directly from Propositions \ref{NH-unst}, \ref{webproof} and \ref{proptsingstable}, 
\end{proof}

\section{Proofs of Theorems B, C and Corollary E}

In this section we intend to discuss the hyperbolic and the parabolic case of the fold-fold singularity in order to complete the characterization of $\s_{0}$.

\subsection{Hyperbolic Fold-Fold}

Let $Z=(X,Y)\in\Or$ be a nonsmooth vector field having a hyperbolic fold-fold point at $p$ such that $S_{X}\pitchfork S_{Y}$ at $p$. Consider the normal coordinates $(x,y,z)$ of $Z$ at $p$ and let $(\ag,\bg,\cg)$ be the normal parameters of $Z$ at $p$. In this case we do not have any orbit of $X$ or $Y$ connecting points of $\s$, therefore the local structural stability of $Z$ at $p$ depends only on the sliding dynamics which is generically characterized in section \ref{slidingsection}.

\begin{proposition}\label{vis}
	Let $Z_{0}=(X_{0},Y_{0})\in\Or$ be a nonsmooth vector field having a visible fold-fold point at $p$ such that $S_{X_{0}}\pitchfork S_{Y_{0}}$ at $p$.  Let $(\ag_{0},\bg_{0},\cg_{0})$ be the normal parameters of $Z_{0}$ at $p$. Then, $Z_{0}$ is locally structurally stable at $p$ if and only if $(\ag_{0},\bg_{0},\cg_{0})\in R_{H}^{1}\cup R_{H}^{2}$.
\end{proposition}
\begin{proof}[Outline]
	The first implication is obvious since $F_{Z_{0}}$ presents bifurcations in $\s^{s}$.
	To prove the converse, let $(\ag_{0},\bg_{0},\cg_{0})$ be the normal parameters of $Z_{0}$ at $p$. Using Implicit Function Theorem we can find a neighborhood $\V$ of $Z_{0}$ in $\Or$ such that every $Z\in\V$ has a hyperbolic fold-fold point $q(Z)$ near $p$ and the normal parameters of $Z$ at $q(Z)$ are close to $(\ag_{0},\bg_{0},\cg_{0})$. 
	
	Now, it is easy to construct a germ homeomorphism $h:\s\rightarrow\s$ carrying sliding orbits of $F_{Z_{0}}$ onto sliding orbits of $F_{Z}$. Extend it to a germ of homeomorphism $h:(M,p)\rightarrow (M,q(Z))$ using the flows in the same way of \cite{F} (Lemma 3, page 271).
\end{proof}

\subsection{Parabolic Fold-Fold}

Let $Z=(X,Y)\in\Or$ be a nonsmooth vector field having an invisible-visible fold-fold point at $p$ such that $S_{X}\pitchfork S_{Y}$ at $p$. Consider the normal coordinates $(x,y,z)$ of $Z$ at $p$, and let $(\ag,\bg,\cg)$ be the normal parameters of $Z$ at $p$. 

Proceeding as in the elliptic case, $Z$ has an involution $\phi_{X}$ associated to the invisible fold of $X$, and remember that in the normal coordinates it is completely known:
$$\phi_{X}(x,y)=(x-2\ag y,-y).$$
Now we use it to study the connections between sliding orbits, when they exist.

\begin{lemma}
	Let $Z=(X,Y)\in\Or$ be a nonsmooth vector field having an invisible-visible fold-fold point at $p$ such that $S_{X}\pitchfork S_{Y}$ at $p$.  Let $(\ag,\bg,\cg)$ be the normal parameters of $Z$ at $p$. Then, $\phi_{X}(S_{Y})\pitchfork S_{Y}$ at $p$ if and only if $\ag\neq 0$.
\end{lemma}
\begin{proof}	
	From Corollary \ref{formaS}, we have that $S_{Y}=\{(g(y),y);\ y\in (-\e,\e)\}$, for some $\e>0$, where $g$ is a smooth function with $g(y)=\er(y^{2})$. Therefore $T_{0}S_{Y}=\textrm{span}\{(0,1)\}$.
	
	On the other hand, $\phi_{X}(S_{Y})=\{(g(y)-2\ag y,-y);\ y\in (-\e,\e)\}$. Then $T_{0}\phi_{X}(S_{Y})=\textrm{span}\{(-2\ag,-1)\}$. The result follows from these expressions
\end{proof}

\begin{lemma}\label{IV}
	Let $Z=(X,Y)\in\Or$ be a nonsmooth vector field having an invisible-visible fold-fold point at $p$ such that $S_{X}\pitchfork S_{Y}$ at $p$.  Let $(\ag,\bg,\cg)$ be the normal parameters of $Z$ at $p$. Then, $\phi_{X}(\s^{us})\cap \s^{ss}= \emptyset$ if and only if $\ag> 0$.
\end{lemma}
\begin{proof}	
	In fact, in these coordinates, $S_{Y}=\{(g(y),y);\ y\in (-\e,\e)\}$, and $\phi_{X}(S_{Y})=\{(g(y)-2\ag y,-y);\ y\in (-\e,\e)\}$, for some $\e>0$, where $g$ is a smooth function with $g(y)=\er(y^{2})$.
	
	Therefore, $T_{0}\phi_{X}(S_{Y})=\textrm{span}\{(-2\ag,-1)\}$. The sliding region $\s^{s}$ is the region delimited by $S_{X}$ and $S_{Y}$.
	
	Since $T_{0}S_{Y}=\textrm{span}\{(0,1)\}$ and $T_{0}S_{X}=\textrm{span}\{(1,0)\}$, it follows that $\phi_{X}(S_{Y})\subset \s^{s}$ if and only if $\ag>0$.
	
	We conclude the proof by noticing that, if $\phi_{X}(S_{Y})\subset \s^{c}$, then $\phi_{X}(\s^{us})\subset \s^{c}$. Nevertheless, if $\phi_{X}(S_{Y})\subset \s^{s}$, then the region delimited by $S_{Y}$ and $\phi_{X}(S_{Y})$ in $\s^{us}$ is carried into the region delimited by $S_{Y}$ and $\phi_{X}(S_{Y})$ in $\s^{ss}$.
\end{proof}

\begin{remark}
	In another words, there exist orbits of $X$ in $M^{+}$ connecting distinct points in the sliding region $\s^{s}$ if and only if $\ag>0$.
\end{remark}

\begin{definition}
	If $\phi:\s\rightarrow\s$ is a diffeomorphism and $F$ is a vector field in $\s$, then define the \textbf{reflected vector field of $F$ by $\phi$} as $\phi^{*}F$.
\end{definition}
\begin{remark}
	The reflected vector field of $F$ by $\phi$ can also be referred as \textbf{transport of $F$ by $\phi$}. 
\end{remark}	

\begin{lemma}
	Let $Z=(X,Y)\in\Or$ be a nonsmooth vector field having an invisible-visible fold-fold point at $p$ such that $S_{X}\pitchfork S_{Y}$ at $p$.  Let $(\ag,\bg,\cg)$ be the normal parameters of $Z$ at $p$. 
	
	Assume that there exist a region $S\subset \s^{us}$ such that $\widetilde{S}=\phi_{X}(S)\subset \s^{ss}$, and suppose that $S$ is maximal with respect to this property. If $2(\ag+\bg)(\ag\bg-\cg)\neq 0$, then $F_{Z}^{N}$ and the transport of $F_{Z}^{N}$ by $\phi_{X}$ are transversal vector fields defined in $\tilde{S}$.
\end{lemma}
\begin{proof}
	Consider $F_{0}=F_{Z}^{N}$ and $F_{1}=\phi^{*}F_{Z}^{N}$, where $\phi_{X}$ is the involution associated to $X$.
	
	Clearly, $F_{0}$ and $F_{1}$ are transversal at $q\in\s$ if and only if $F_{0}(q)$ and $F_{1}(q)$ are linearly independent vectors. 
	
	Considering the normal coordinates $(x,y,z)$ at $p$, define the following function: 
	\begin{equation}\label{det}
	D(x,y)=\det\left(\begin{array}{c}
	F_{0}(x,y)\\
	F_{1}(x,y)
	\end{array}\right).
	\end{equation}
	
	Notice that  $D(x,y)\neq 0$ if and only if $F_{0}$ and $F_{1}$ are transversal at $(x,y)$.
	
	Now, we use the expressions of the vector field in these coordinates to derive an approximation for the function $D$.

	Since $\phi_{X}$ is a linear involution, it follows that $\phi_{X}^{-1}=\phi_{X}$ and $d\phi_{X}=\phi_{X}$, therefore:
	
	\begin{equation}
	\begin{array}{lcl}
	F_{1}(x,y) &=& d\phi_{X}(F_{Z}^{N}(\phi_{X}^{-1}(x,y)))\\
	&=& \phi_{X}(F_{Z}^{N}(\phi_{X}(x,y)))\\
	\end{array}
	\end{equation}
	
	In order to compute $D$, we must analyze the influence of the higher order terms in the computation of $F_{Z}^{N}$. From Proposition \ref{FFNF_prop}, we have that:
	
	\begin{equation}\label{NFIV2}
	X(x,y,z)=\left(\begin{array}{c}
	\alpha\\
	1\\
	-y
	\end{array}\right)\textrm{ and }Y(x,y,z)=\left(\begin{array}{c}
	\gamma + \widetilde{F}(x,y,z)\\
	\beta + \widetilde{G}(x,y,z)\\
	x + \widetilde{H}(x,y,z)
	\end{array}\right),
	\end{equation}
	where $\widetilde{F}(x,y,z)=\er(|(x,y,z)|)$, $\widetilde{G}(x,y,z)=\er(|(x,y,z)|)$ and $\widetilde{H}(x,y,z)=\er(|(x,y,z)|^{2})$.
	
	Hence, the sliding vector field is given by:
	$$	F_{Z}^{N}(x,y)=\left(\begin{array}{cc}
	\ag & \cg\\
	1 & \bg
	\end{array}\right)\cdot\left(\begin{array}{c}
	x\\
	y
	\end{array}\right) +\left(\begin{array}{c}
	\ag H(x,y) + y F(x,y)\\
	H(x,y)+ yG(x,y)
	\end{array}\right),$$	
	where $F(x,y)=\widetilde{F}(x,y,0)=\er(|(x,y)|)$, $G(x,y)=\widetilde{G}(x,y,0)=\er(|(x,y)|)$ and $H(x,y)=\widetilde{H}(x,y,0)=\er(|(x,y)|^{2})$.
	
	Using the expression of $F_{Z}^{N}$ and $\phi_{X}(x,y)(x-2\ag y,-y)$, we obtain:	
	$$D(x,y)= y^{2}[-2(\ag+\bg)(\ag\bg-\cg) +P_{1}(x,y)]+yP_{2}[x,y],$$
	where $P_{1}(x,y)=\er(|(x,y)|)$ and $P_{2}=\er(|(x,y)|)$.
	
	In addition, $P_{2}(x,0)\equiv 0$, hence we can use Malgrange Preparation Theorem to find a smooth function $P_{3}$, such that $P_{2}(x,y)=yP_{3}(x,y)$ and $P_{3}(x,y)=\er(|(x,y)|)$.
	
	With this, we conclude that 
	
	$$D(x,y)=[-2(\ag+\bg)(\ag\bg-\cg)+ P(x,y)]y^{2},$$
	where $P(x,y)=\er(|(x,y)|)$.
	
	
	Now, if $(\ag+\bg)(\ag\bg-\cg)\neq 0$, then the $x$-axis is the only solution of $D(x,y)=0$, near the origin. Therefore the vector fields  $F_{0}$ and $F_{1}$ are transversal in the region $S\cup\widetilde{S}$, since it does not contain points of the $x$-axis.
	
\end{proof}

\begin{remark}
	Notice that, in the curves $\ag+\bg=0$ and $\ag\bg=\cg$, the higher order terms may produce curves in $S\cup\widetilde{S}$ where the vector fields are not transversal, and they can be broken by small perturbations (making $\ag+\bg\neq 0$ or $\ag\bg\neq\cg$). Clearly, this situation imply in the instability of the system.
\end{remark}

\begin{lemma}
	Let $Z=(X,Y)\in\Or$ be a nonsmooth vector field having an invisible-visible fold-fold point at $p$ such that $S_{X}\pitchfork S_{Y}$ at $p$.  Let $(\ag,\bg,\cg)$ be the parameters given by Proposition \ref{FFNF_prop} associated to $Z$ at $p$.	If $2\ag(\ag+\bg)-\cg\neq0$, then $F_{Z}^{N}$ is transversal to $\phi_{X}(S_{Y})$ in $\s^{s}$.
\end{lemma}
\begin{proof}
	In the coordinates of Proposition \ref{FFNF_prop}, we have that $S_{Y}=\{(g(y),y,0);\ y\in(-\e,\e)\}$, for $\e>0$ sufficiently small, where $g$ is a $\Cr$ function such that $g(y)=\er(y^{2})$.
	
	Therefore  $\phi_{X}(S_{Y})=\{(g(y)-2\ag y,-y);\ y\in (-\e,\e)\}$. Since $\phi_{X}(S_{Y})$ is tangent to the curve $\gamma(y)=(-2\ag y,-y)$ at the origin, it is sufficient to prove that $F_{Z}^{N}$ is transversal to $\gamma$.
	
	Clearly, $F_{Z}^{N}$ is transversal to $\gamma$ at $\gamma(y)$ if and only if:
	
	\begin{equation}\label{T}
	T(y)=F_{Z}^{N}(\gamma(y))\cdot(\gamma'(y))^{\perp}\neq 0.
	\end{equation}	
	
	Now, we use the expression of $F_{Z}^{N}$ in these coordinates to obtain an approximation of $T$. In fact, 	
	$$F_{Z}^{N}(\gamma(y))=F_{Z}^{N}(-2\ag y, -y)=(-2\ag^{2}y- \cg y, -2\ag y-\bg y)+ \er(y^{2})$$
	and
	$$(\gamma'(y))^{\perp}= (-2\ag,-1)^{\perp}=(1,-2\ag).$$	
	
	Substituting these expressions in \ref{T}, we obtain:
	$$T(y)=[2\ag(\ag+\bg)-\cg]y+ \er(y^{2})$$
	
	Therefore, if the condition $2\ag(\ag+\bg)-\cg\neq 0$ is assumed and $y\neq 0$ then $F_{Z}^{N}$ is transversal to $\phi_{X}(S_{Y})$. Since $\s^{s}$ does not contain points where $y\neq0$ (because they belong to $S_{X}$), the result follows.
\end{proof}

\begin{remark}
	In the curve $2\ag(\ag+\bg)-\cg=0$, the higher order terms can be used to produce a curve such that $F_{Z}^{N}$ is tangent to $\p_{X}(S_{Y})$ in every point. Such structurally unstable phenomena have to be avoided.
\end{remark}

\begin{proposition}\label{vis-inv}
	Let $Z_{0}=(X_{0},Y_{0})\in\Or$ be a nonsmooth vector field having an invisible-visible fold-fold point at $p$ such that $S_{X_{0}}\pitchfork S_{Y_{0}}$ at $p$. Let $(\ag_{0},\bg_{0},\cg_{0})$ be the normal parametersof $Z_{0}$ at $p$. Then, $Z_{0}$ is locally structurally stable at $p$ if and only if:
	\begin{enumerate}
		\item $(\ag_{0},\bg_{0},\cg_{0})\in \displaystyle\cup_{i=1}^{4} R^{i}_{P}$;
		\item $\ag_{0}\neq 0$;
		\item $2\ag_{0}(\ag_{0}+\bg_{0})-\cg_{0}\neq0$;
		\item $\ag_{0}+\bg_{0}\neq 0$, if $\ag_{0}>0$.
	\end{enumerate}
	Moreover, there exist only eleven topologically distinct classes of local structural stable systems at invisible-visible fold-fold points.
\end{proposition}

\begin{proof}[Outline]
	Proceeding as is the proof of Theorem \ref{vis}. Consider the neighborhood $\V$ of $Z_{0}$ such that the correspondent parameters $(\ag,\bg,\cg)$ of any $Z\in\V$ are in the same region of $(\ag_{0},\bg_{0},\cg_{0})$.
	
	Let $Z=(X,Y)\in\V$. If there is no orbits of $X$ connecting points of $\s^{ss}$ and $\s^{us}$, then the proof can be done in the following way. We omit some details in this case, since it is very similar to the visible case.
	\begin{itemize}
		\item Construct $h: \s^{s}(Z_{0})\rightarrow \s^{s}(Z)$ carrying orbits of $F_{0}$ onto orbits of $F_{Z}$. In addition extend it to $S_{X_{0}}\cup S_{Y_{0}}$ via limit. Hence $h(S_{X_{0}})=S_{X}$ and $h(S_{Y_{0}})=S_{Y}$;
		\item For each $p\in \s\setminus S_{X_{0}}$, there exists $t_{0}(p)\neq 0$ such that $\p_{X_{0}}(t_{0}(p),p)\in \s$. Similarly, there exists $t(p)\neq 0$ for the vector field $X$;
		\item If $p\in \s^{s}$, then $h(p)$ is already defined. Assume that $p\in \s^{c}$. If $\p_{X_{0}}(t_{0}(p),p)\in \s^{s}$, then define:
		$$h(p)= \p_{X}(-t(\p_{X_{0}}(t_{0}(p),p)), h(\p_{X_{0}}(t_{0}(p),p))).$$
		\item Using Tietze Extension Theorem, we can extend $h$ over $\s^{c}$;
		\item Now, using the same idea of the third item, we can extend it to the whole $\s$;
		\item Extend it to $M^{+}$ using the flow of $X_{0}$, $X$ and $h:\s\rightarrow \s$;
		\item Following the same idea of the hyperbolic case, extend it to $M^{-}$;
		\item Hence we construct a germ of homeomorphism $h:M\rightarrow M$ at $p$, with $h(p)=q(Z)$, which is an equivalence between $Z_{0}$ and $Z$. Then $Z_{0}$ is locally structurally stable at $p$.	 
	\end{itemize}	
	
	Suppose that there exists a connection between $\s^{ss}$ and $\s^{us}$ for $Z_{0}$ and $Z$. Denote by $S_{0}$ and $S$, the regions of $\s^{s}$ exhibiting connections.	
	
	From the previous Lemmas of this subsection, it is possible to say that $F_{0}$ and $\phi_{X_0}^{*}F_{0}$ are transversal in each point of $S_{0}$, and the same works for $F_{Z}$ and $\phi_{X}^{*}F_{Z}$ in $S$.
	
	Therefore, the orbits of $F_{0}$ and $\phi_{X_0}^{*}F_{0}$ define a coordinate system in $S_{0}$, such as the orbits of $F_{Z}$ and $\phi_{X}^{*}F_{Z}$ in $S$.
	
	Hence, let $h$ be a function carrying $S_{Y_{0}}$ onto $S_{Y}$, and $h(0)=0$. Now we can use these coordinate systems to extend $h:S_{0}\rightarrow S$. Moreover, it satisfies:	
	$$h\circ\phi_{X_{0}}=\phi_{X}\circ h.$$
	
	By the transversality of $F_{0}$ to $\phi_{X_0}(S_{Y_{0}})$ (resp. $F_{Z}$ to $\phi_{X}(S_{Y})$), it is possible to extend $h$ on $\s^{s}(Z_{0})$ using the sliding orbits. Then we have a homeomorphism $h:\s^{s}(Z_{0})\rightarrow \s^{s}(Z)$ carrying sliding orbits onto sliding orbits.
	
	By construction, if $x\in S$, then $\phi_{X}(h(x))=h(\phi_{X_0}(x))$. With this, we can use the same idea from the previous case without connections to extend such map to a germ of homeomorphism $h:M\rightarrow M$ at $p$, with $h(p)=q(Z)$, which is a topological equivalence between $Z_{0}$ and $Z$ at $p$. 	
\end{proof}	

\subsection{Proof of Theorem B}

Notice that $Z$ satisfies condition $\s(H)$ at $p$ if, and only if, the normal parameters $(\ag,\bg,\cg)$ of $Z$ at $p$ satisfy the hypotheses of Proposition \ref{vis}.

Moreover, $Z$ satisfies condition $\s(P)$ at $p$ if, and only if, the normal parameters $(\ag,\bg,\cg)$ of $Z$ at $p$ satisfy the hypotheses of Proposition \ref{vis-inv}.

The result follows directly from Propositions \ref{vis}, \ref{vis-inv}. 

\subsection{Proof of Theorem C}
From Proposition \ref{vishik} it follows that $\s_{0}\subset \s(G)$.

The result follows from Theorem \ref{generic} and from Theorems A and B.

\subsection{Proof of Corollary E}
From the characterization of $\s_{0}$, we can see that $\s(G)$, $\s(R)$, $\s(H)$, $\s(P)$ are open dense sets in $\Or$.

Nevertheless, we also prove that $\s(E)$ is not residual in $\Or$. Therefore, it follows that $\s_{0}\cap \s(E)$ is open dense in $\s(E)$ and $\s(E)$ is the biggest set with this property.

\section{Acknowledgments}
This research has been partially supported by FAPESP Thematic Project (2012/18780-0) and FAPESP PhD Scholarship (2015/22762-5).

\bibliographystyle{amsplain}


\end{document}